\documentclass[11pt,fleqn,a4paper]{scrartcl} 

\usepackage[a4paper,left=4cm,right=4cm,top=3cm,bottom=3cm]{geometry}

\usepackage{verbatim}
\usepackage[english]{babel}
\usepackage[latin1]{inputenc} 
\usepackage{amssymb,amsmath,amsthm,amsfonts}
\usepackage{mathrsfs}

\usepackage{xcolor}

\usepackage[T1]{fontenc} 
\usepackage{ae}

\usepackage{url}        
\usepackage{enumerate}

\usepackage[all]{xy}

\usepackage{makeidx}    
\usepackage{showidx}

\usepackage{fancyhdr}

\renewcommand{\subsectionmark}[1]{} 
\fancypagestyle{plain}{
	\fancyhead{} 
	 
}

\theoremstyle{plain}
\newtheorem{theorem}{Theorem}[section]
\newtheorem{proposition}[theorem]{Proposition}
\newtheorem{lemma}[theorem]{Lemma}
\newtheorem{corollary}[theorem]{Corollary}

\newtheorem*{theoremA}{Theorem A}
\newtheorem*{theoremB}{Theorem B}
\newtheorem*{theoremC}{Theorem C}
\newtheorem*{theoremD}{Theorem D}

\theoremstyle{definition}
\newtheorem{definition}[theorem]{Definition}

\theoremstyle{remark}

\newcommand{\N}{\mathbb N}

\newcommand{\R}{\mathbb R} 
\newcommand{\C}{\mathbb C}

\newcommand{\K}{\mathbb K}

\renewcommand{\epsilon}{\varepsilon}

\newcommand\Exp{\mathop{\mathrm{E}}\nolimits}

\newcommand\Lin{\mathrm{Lin}}

\newcommand\id{\mathrm{id}}

\newcommand{\set}[2]{ \left\{  #1  :  #2  \right\} }

\newcommand{\smset}[1]{\{ #1 \}}

\newcommand{\func}[5]{#1 \colon #2 \longrightarrow #3 : #4 \mapsto #5}
\newcommand{\smfunc}[3]{#1 \colon #2 \longrightarrow #3}
\newcommand{\nnfunc}[4]{#1 \longrightarrow #2 : #3 \mapsto #4}
\newcommand{\nnsmfunc}[2]{#1 \longrightarrow #2}

\newcommand{\Func}[5]{
\begin{array}{rl}
   #1 : #2 & \longrightarrow #3   \\
        #4 & \longmapsto     #5   \\
\end{array}
}

\newcommand{\nnFunc}[4]{
\begin{array}{rl}
        #1 & \longrightarrow #2   \\
        #3 & \longmapsto     #4   \\
\end{array}
}

\newcommand{\oBallin}[3]{\mathrm{B}_{#1}^{#2}\left(#3 \right)}

\newcommand{\V}{\mathcal{V}}

\newcommand{\Vdelta  }{\V\left(\delta_1, \delta_2, \ldots \right) }

\newcommand{\norm}[1]{\ensuremath{\left\| #1 \right\|}}

\newcommand{\biggnorm}[1]{\ensuremath{\bigg\| #1 \bigg\|}}

\newcommand{\Biggnorm}[1]{\ensuremath{\Bigg\| #1 \Bigg\|}}
\newcommand{\opnorm}[1]{\ensuremath{  \norm{  #1 }_{\mathrm{op}} }\!  }
\newcommand{\infnorm}[1]{\ensuremath{  \norm{  #1 }_\infty }}
\newcommand{\supnorm}[1]{\infnorm{#1}}
\newcommand{\maxnorm}[1]{\infnorm{#1}}
\newcommand{\Xnorm}[1]{\ensuremath{  \norm{  #1 }_{X} }}

\newcommand{\XBiggnorm}[1]{\ensuremath{  \Biggnorm{  #1 }_{X} }}

\newcommand{\nnorm}[1]{\ensuremath{  \norm{  #1 }_{n} }}
\newcommand{\gnorm}[1]{\ensuremath{  \norm{  #1 }_{\g} }}
\newcommand{\Dnorm}[1]{\ensuremath{  \norm{  #1 }_{\mathrm{D}} }}

\newcommand{\abs}[1]{\ensuremath{| #1 |}}

\newcommand{\g}{\mathfrak{g}}
\renewcommand{\L}{\mathbf{L}}
\newcommand{\brac}{[\cdot,\cdot]}
\renewcommand{\Exp}{\mathrm{Exp}}
\newcommand{\Ad}{\mathrm{Ad}}

\newcommand{\generatedby}[1]{\ensuremath{ \left\langle   #1  \right\rangle  }}
\newcommand{\component}[1]{\left(#1\right)_0}

\newcommand{\seqn}[1]{\left(#1\right)_{n\in \N}}

\newcommand{\lone}{\ell^1}

\newcommand{\Di}[2]{D_{#1}\!\left(#2\right)}
\newcommand{\Dinf}[1]{\Di{\infty}{#1}}

\newcommand{\DV}{\mathbf{z}}
\newcommand{\Din}[1]{\sum_{n\in\N} #1 \cdot n^{-\DV}}
\newcommand{\DiN}[1]{\sum_{N\in\N} #1 \cdot N^{-\DV}}
\newcommand{\DiOne}[1]{            #1 \cdot 1^{-\DV}}

\newcommand{\Halfo}[1]{\mathbb{H}_{#1}}
\newcommand{\Halfc}[1]{\overline{\mathbb{H}}_{#1}}

\newcommand{\DiNorm}[2]{\norm{#1}_{(#2)}}
\newcommand{\biggDiNorm}[2]{\biggnorm{#1}_{(#2)}}

\newcommand{\Cont}[2]{\mathcal{C}\left( #1 , #2 \right)}
\newcommand{\Contb}[2]{\mathcal{C}_b\left( #1 , #2 \right)}

\newcommand{\BoundOp}[1]{\mathcal{L}\left(#1\right)}
\newcommand{\BoundOpNorm}[1]{\left( \BoundOp{#1}  , \opnorm{\cdot}   \right)}

\newcommand{\injepsilon}{\epsilon_\circ}

\newcommand{\cl}[1]{\left[ #1 \right]_\sim}

\newcommand{\GL}[1]{ \mathrm{GL}\left(#1\right) }

\newcommand{\DiffGermK}{\mathrm{DiffGerm}(K,X)}
\newcommand{\GermDiffK}{\DiffGermK}

\newcommand{\EndGermK}{\mathrm{EndGerm}(K,X)}
\newcommand{\GermEndK}{\EndGermK}

\newcommand{\GermK}{\mathrm{Germ}(K,X)_K}
\newcommand{\HolbK}[1]{\mathrm{Hol}_\mathrm{b} \left(U_{#1},X\right)_K  }
\newcommand{\BConeK}[1]{\mathrm{BC}_\C^1 \left(U_{#1},X\right)_K  }
\newcommand{\glchart}{\Phi}
\newcommand{\Inversion}{\mathbf{i}}

\newcommand{\BCH}{\emph{BCH}}

\newcommand{\Frechet}{Fréchet}
\newcommand{\frechet}{Fréchet}

\usepackage{cite}

\title{Analytic Mappings Between LB-spaces and Applications in Infinite-Dimensional Lie Theory}
\author{Rafael Dahmen}

\date{July 22, 2008
}

\begin{document}
\maketitle
\section*{Abstract}
\markright{\textsc{Abstract}}

We give a sufficient criterion for complex analyticity of nonlinear maps defined on direct limits of normed spaces. This tool is then used to construct new classes of (real and complex) infinite dimensional Lie groups: The group $\DiffGermK$ of germs of analytic diffeomorphisms around a compact set $K$ in a Banach space $X$ and the group $\bigcup_{n\in\N}G_n$ where the $G_n$ are Banach Lie groups.

\tableofcontents

\subsection*{Acknowledgement}
The author was supported
by the German Research Foundation (DFG),
grant GL 357/7-1.
The studies are part of the author's Ph.D.-project,
advised by Helge Gl\"{o}ckner.

\section*{Introduction}
\markright{\textsc{Introduction}}
\addcontentsline{toc}{section}{Introduction}

An infinite dimensional complex analytic Lie group is a group which is at the same time a complex analytic manifold modelled on some locally convex topological vector space such that the group operations are analytic. To construct such Lie groups, it is useful to have tools at hand ensuring the complex analyticity of nonlinear mappings between locally convex spaces. This text provides a sufficient criterion for complex analyticity in the case where the domain is an LB-space, i.e. a locally convex direct limit of an ascending sequence of Banach spaces:

\begin{theoremA}[Complex analytic mappings defined on LB-spaces]				 	\label{thm_LB}
 Let $E$ be a $\C$-vector space that is the union of the increasing sequence of subspaces $(E_n)_{n\in \N}$. Assume that a norm $\norm{\cdot}_{E_n}$ is given on each $E_n$ such that all bonding maps
\[ \func{i_n }{E_n}{E_{n+1}}{x}{x}
\]
are continuous and have an operator norm at most $1$. We give $E$ the locally convex direct limit topology and assume that it is Hausdorff.
Let $R>0$ and let $U:=\bigcup_{n\in \N} \oBallin{R}{E_n}{0}$ be the union of all open balls with radius $R$ around $0$.
Let $\smfunc{f}{U}{F}$ be a function defined on $U$ with values in a locally convex space $F$, such that each $f_n := f|_{\oBallin{R}{E_n}{0}}$ is $\C$-analytic and \emph{bounded}. Then $f$ is $\C$-analytic.
\end{theoremA}

Note that the statement cannot be generalized to direct limits of \Frechet{} spaces. There even exist bilinear mappings that are continuous on each step $E_n$ but fail to be continuous on the limit. Similar pathologies arise when looking at uncountable direct limits of normed spaces.

In Section \ref{sec_DIFFEO} we will use Theorem~A
to construct Lie groups of germs of diffeomorphisms:

\begin{theoremB}[Lie group of germs of diffeomorphisms]					\label{thm_DIFFEO}
 Let $\K\in\{\R,\C\}$ and let $X$ be a Banach space over $\K$. Let $K\subseteq X$ be a non-empty compact subset of $X$. Consider the group
\[
 \DiffGermK :=  \set{    \eta  }{  \begin{array}{c}
 								 \eta \hbox{ is a $C^\omega$-diffeomorphism between open} \\
								\hbox{ neighborhoods of $K$ and $\eta|_K=\id_K$}
\end{array}
 }/_\sim,
\]
 where two diffeomorphisms $\eta_1, \eta_2$ are considered equivalent, $\eta_1\sim\eta_2$ if they coincide on a common neighborhood of $K$.
 Then $\DiffGermK$ can be turned into a Lie group modelled on a compactly regular LB space.
\end{theoremB}
If $X$ is of finite dimension, this is known and can be found in \cite{MR2310802}. If in addition $K=\smset{0}$, the Lie group structure was first constructed by Pisanelli in \cite{MR0436234}.

Theorem~A also allows us to construct Lie groups which can be written as an ascending union of Banach Lie groups. In Section \ref{sec_UNION} we will prove the following result:

\begin{theoremC}[Unions of Banach Lie groups]						\label{thm_UNION}
Let $G_1\subseteq G_2 \subseteq \cdots$ be analytic Banach Lie groups over $\K\in\{\R,\C\}$, such that all inclusion maps $\smfunc{i_n}{G_n}{G_{n+1} }$ are analytic group homomorphisms. Set $G:= \bigcup_{n\in\N}G_n$. Assume that the following hold:
\begin{itemize}
 \item [(a)] For each $n\in\N$ there is a norm $\nnorm{\cdot}$ on $\g_n:=\L(G)$ defining its topology, such that $\nnorm{[x,y]}\leq \nnorm{x}\nnorm{y}$ for all $x,y\in \g_n$ and such that the bounded operator $\smfunc{\L(i_n)}{\g_n}{\g_{n+1} }$ has operator norm at most $1$.
 \item [(b)] The locally convex direct limit topology on $\g:= \bigcup_{n\in\N} \g_n$ is Hausdorff.
 \item [(c)] The map $\exp_G := \smfunc{\bigcup_{n\in\N}\exp_{G_n} }{\g}{G}$ is injective on some $0$-neigh\-bor\-hood
\end{itemize}
Then there exists a unique $\K$-analytic Lie group structure on $G$ which makes $\exp_G$ a local $C^\omega_\K$-diffeomorphism at $0$.
\end{theoremC}

There are many applications of Theorem~C,
 for example the following theorem to be proved in Section \ref{sec_DIRICHLET}, in which all details of the construction can be found.
\begin{theoremD}[Lie groups associated with Dirichlet series]
 Let $G$ be a complex Banach Lie group with Lie algebra $\g$. Let $\Di{s}{\g}$ be the Banach Lie algebra of $\g$-valued Dirichlet Series which converge absolutely on the complex half plane $\Halfc{s}$ and $\Di{s}{G}$ be the Banach Lie group generated by $\set{\exp_G\circ \gamma}{\gamma\in\Di{s}{\g}}$.
Then the group $\Dinf{G}:=\bigcup_{s\in\R}\Di{s}{G}$ can be made into a Lie group modeled on the direct limit Lie algebra $\Dinf{\g}=\bigcup_{s\in\R}\Di{s}{g}$.
\end{theoremD}

\section{Preliminaries and notation}							\label{sec_PRELIMINARIES}

\subsection*{Analytic mappings between locally convex spaces}

\begin{definition}[Complex analytic mappings]					\label{def_complex_analytic}
 Let $E$ and $F$ be locally convex spaces over $\C$, let $\smfunc{f}{U}{F}$ be a mapping from an open subset $U\subseteq E$ with values in $F$.
 We say that $f$ is \emph{complex analytic} or $C^\omega_\C$ if it is continuous and if it admits locally a power series expansion around each point $a\in U$, which means there exist continuous homogeneous polynomials $p_k$ of degree $k$ such that
\[
 f(x)=\sum_{k=0}^\infty p_k(x-a) 
\]
for all $x$ in a neighborhood of $a$.
\end{definition}

\begin{definition}[Real analytic mappings]
 Let $E$ and $F$ be locally convex spaces over $\R$ and let $E_\C$ and $F_\C$ be their complexifications. A mapping $\smfunc{f}{U}{F}$ from an open subset $U\subseteq E$ with values in $F$ is called \emph{real analytic} or $C^\omega_\R$ if it extends to a complex analytic $F_\C$-valued map on an open
neigh\-bor\-hood of $U$ in the complexification $E_\C$.
\end{definition}

The composition of $\K$-analytic maps (for $\K\in\{\R,\C\}$) is again $\K$-analytic (see \cite[Theorem 6.4]{MR0313811}), therefore we may define $C^\omega$-manifolds and $C^\omega$-Lie groups (see Definition \ref{def_lie_group}).

There is an easy characterization of complex analyticity, that can be found in \cite[Theorem 6.2]{MR0313811} :
\begin{lemma}[Gateaux Analyticity]						\label{lemma_gateaux}
 Let $\smfunc{f}{U\subseteq E}{F}$ be a function defined on some open subset of a complex locally convex space $E$. Then $f$ is $C^\omega_\C$ if and only if it is continuous and \emph{Gateaux analytic}, which means that for every point $a\in U$ and every vector $b\in E$ there exists an $\epsilon>0$ such that the function
\[
 \nnfunc{ \oBallin{\epsilon}{\C}{0}  }{F}{z}{f(a+zb)}
\]
 is complex analytic.
\end{lemma}

An analytic map between complex \emph{normed} spaces has the property that all \frechet{} derivatives $\smfunc{f^{(k)}}{U}{\Lin^k(E,F)}$ exist and are analytic. The symmetric $k$-linear maps from Definition \ref{def_complex_analytic} can be written as  $\beta_k = \frac{1}{k!}f^{(k)}(a)\in \Lin^k(E,F)$.

In this context of normed spaces there is another useful

\begin{lemma}[Absolute convergence of bounded power series]			\label{lemma_absolute}
 Let $E$ and $F$ be complex normed vector spaces and let $\smfunc{f}{\oBallin{R}{E}{a}}{F}$ be a \emph{bounded} complex analytic map with the following power series expansion: 
\[
  f(a+x) = \sum_{k=0}^\infty \beta_k(x,\ldots,x)
\]
where the $\smfunc{\beta_k}{E^k}{F}$ are continuous symmetric $k$-linear maps.
Then for all $r<\frac{R}{2e}$ the following series converges and can be estimated as shown:
\[
	  \sum_{k=0}^\infty \opnorm{\beta_k} r^k \leq \frac{R}{R-2er} \cdot \maxnorm{f}.
\]
Here $\maxnorm{f}:=\sup\set{\norm{f(x)}}{x\in \oBallin{R}{E}{a}}$ and $e=2.718281828\ldots$.
\end{lemma}

This lemma turns out to be extremely helpful throughout this paper. In Section \ref{sec_DIFFEO} we will need a further generalization:
\begin{lemma}[Absolute convergence of families of bounded power series]		\label{lemma_absolute_family}
Let $K\subseteq E$ be a nonempty subset of a complex normed vector space $E$. Let $U:=K+\oBallin{R}{E}{0}= \bigcup_{a\in K}\oBallin{R}{E}{a}$ be a union of open balls with fixed radius \hbox{$R>0$}. Now, consider a set $M$ of bounded $C^\omega_\C$-mappings from $U$ to a normed space $F$ such that $\sup_{f\in M} \supnorm{f}<\infty$.
Then we have for all $r<\frac{R}{2e}$ the following estimate:
\[
	  \sum_{k=0}^\infty \sup_{ \substack{f\in M\\a\in K} } \frac{\opnorm{f^{(k)}(a)}}{k!} r^k \leq \frac{R}{R-2er} \cdot \sup_{f\in M} \supnorm{f}.
\]
\end{lemma}
This clearly implies Lemma \ref{lemma_absolute} by taking $K:=\smset{a}$ and $M:=\smset{f}$.

\begin{proof}[Proof of Lemma \ref{lemma_absolute_family}]
 Without loss of generality, we may assume that $F$ is a Banach space. Let $f\in M$ and $a\in K$ be given. Let $v\in E$ be a vector of norm $\norm{v}_E =1$. Furthermore let $s<R$ be a fixed number. Then we define the following function
\[
 \func{g}{\oBallin{R}{\C}{0} }{F}{z}{f(a+zv)}
\]
Note that $g$ depends on the choices of $f,a$ and $v$. It is possible to expand $g$ into a power series:
\[
 g(z)=\sum_{k=0}^\infty \frac{1}{k!}f^{(k)}(a)(zv,\ldots,zv)=\sum_{k=0}^\infty \frac{1}{k!}f^{(k)}(a)(v,\ldots,v)\cdot z^k
\]
Using the Cauchy formula (see e.g. \cite[Corollary 3.2]{MR0313811}), we can write the coefficients of this power series as a complex integral:
\[
 \frac{1}{k!}f^{(k)}(a)(v,\ldots,v) = \frac{1}{2\pi i}\int_{\abs{z}=s}\frac{g(z)}{z^{k+1} }dz
\]
Now we can estimate the norm of $f^{(k)}(a)(v,\ldots,v)$:
\begin{align*}
 \norm{f^{(k)}(a)(v,\ldots,v)}_F &\leq k! \norm{\frac{1}{2\pi i}\int_{\abs{z}=s}\frac{g(z)}{z^{k+1} }dz}_F
		\leq k! \frac{1}{2\pi} 2\pi s \frac{\supnorm{g}}{s^{k+1} }
	\\&	\leq \frac{k!}{s^k}\supnorm{f}.
\end{align*}
Since $v$ was arbitrary, this gives us an estimate for the norm of the homogeneous polynomial
\[
 \opnorm{v \mapsto f^{(k)}(a)(v,\ldots,v)  }\leq \frac{k!}{s^k}\supnorm{f}.
\]
By the formula in \cite[Proposition 1.1]{MR0313810} this implies an upper bound for the norm of the corresponding $k$-linear map:
\[
 \opnorm{f^{(k)}(a)}\leq \frac{(2k)^k}{k!}\cdot  \frac{k!}{s^k}\supnorm{f}\leq (2e)^k \frac{k!}{s^k}\supnorm{f}.
\]
Since $s<R$, $a\in K$ and $f\in M$ were arbitrary, we obtain the following:
\[
 \sup_{ \substack{f\in M\\a\in K} }\opnorm{f^{(k)}(a)}\leq (2e)^k \frac{k!}{R^k}\cdot\sup_{f\in M}\supnorm{f}.
\]
Now, we multiply both sides of the inequality by $\frac{r^k}{k!}$ and sum up:
\[
  \sum_{k=0}^\infty \sup_{ \substack{f\in M\\a\in K} } \frac{\opnorm{f^{(k)}(a)} }{k!} r^k 
		\leq \sum_{k=0}^\infty \left(\frac{2er}{R}\right)^k \cdot \sup_{f\in M} \supnorm{f}.
\]
Since $r$ was assumed to be strictly less that $\frac{R}{2e}$, the series $\sum_{k=0}^\infty \left(\frac{2er}{R}\right)^k$ converges to $\frac{1}{1-2er/R}=\frac{R}{R-2er}$. This finishes the proof.
\end{proof}

\subsection*{Generating Lie groups from local data}
The definitions of analytic manifolds and Lie groups are analogous to the finite dimensional case:
\begin{definition}										\label{def_lie_group}
 \begin{itemize}
 \item [(a)] Let $E$ be locally convex $\K$-vector space with $\K\in\smset{\R,\C}$. A \emph{$C^\omega_\K$-manifold modelled on $E$} is a Hausdorff space $M$ together with a maximal set \emph{(atlas)} of homeomorphisms (\emph{charts}) $\smfunc{\phi}{U_\phi}{V_\phi}$ where $U_\phi\subseteq M$ and $V_\phi \subseteq E$ are open subsets, such that the transition maps $\smfunc{\phi \circ \psi^{-1}}{ \psi\left(U_\phi\cap U_\psi\right)  }{\phi\left(U_\phi\cap U_\psi\right)}$ are analytic.
 \item [(b)] Continuous mappings between analytic manifolds are called \emph{analytic} if they are analytic after composition with suitable charts.
 \item [(c)] An \emph{analytic Lie group} is a group which is at the same time an analytic manifold such that the group operations are analytic.
 \end{itemize}
\end{definition}

\begin{proposition}									\label{prop_local_data}	
 Let $\K\in\smset{\R,\C}$. Let $G$ be a group, let $V\subseteq G$ be a subset of $G$ carrying a $C^\omega_\K$-manifold structure. Let $U \subseteq V$ be an open symmetric subset containing $1_G$ such that $U\cdot U \subseteq V$. We assume that
\begin{itemize}
 \item[(i)] the multiplication map $\nnfunc{U\times U}{V}{(x,y)}{x\cdot y}$ is analytic,
 \item[(ii)] the inversion map      $\nnfunc{U}{U}{x}{x^{-1}}$ is analytic,
 \item[(iii)] for every $g\in G$ there is an open subset $W_g\subseteq U$ such that the conjugation map $\nnfunc{W_g}{U}{x}{gxg^{-1}}$ is analytic.
\end{itemize}
Then there exists a unique $C^\omega_\K$-Lie group structure on $G$ such that $U$ is an open subset of $G$ with the given manifold structure.
If $G$ is generated by $U$ then  (i) and (ii) imply (iii).
\end{proposition}
\begin{proof}
 See \cite[Chapter III, \S 1.9, Proposition 18]{MR1728312} for the case of a Banach Lie group.
\end{proof}

\begin{corollary}									\label{cor_local_data}	
 Let $\g$ be a locally convex Lie algebra and let $U$ and $V$ be two symmetric $0$-neigh\-bor\-hoods in $\g$ such that the \emph{BCH}-series converges on $U\times U$ and defines a $C^\omega$-map $\smfunc{\ast}{U\times U}{V}$. Let $\smfunc{\Phi}{V}{H}$ be an injective map into a group $H$ such that $\Phi(x*y)=\Phi(x)\cdot \Phi(y)$. Then there exists a unique $C^\omega$-Lie group structure on $G:=\generatedby{\Phi(U)}$ such that $\smfunc{\Phi|_U}{U}{\Phi(U)}$ becomes a diffeomorphism onto an open subset.

Furthermore, the topological isomorphism $\smfunc{T_0\Phi}{\g}{\L(G)}$ is an isomorphism of locally convex Lie algebras and after identifying $\g$ with $\L(G)$, we obtain that $G$ admits an exponential function and we have $\exp_G|_U = \Phi|_U$.
\end{corollary}
\begin{proof}
 This is an easy consequence of Proposition \ref{prop_local_data}
\end{proof}

\subsection*{Locally convex direct limits}						\label{subsec_direct_limit}
Direct limits will only occur in the following situation:
Let $E_1 \subseteq E_2 \subseteq \cdots$ be an increasing sequence of normed $\K$-vector spaces such that all bonding maps $\smfunc{i_n}{E_n}{E_{n+1}}$ are continuous (i.e.{}\ bounded operators). Then we define the locally convex direct limit of the sequence $\seqn{E_n}$ as the union $E:= \bigcup_{n=1}^\infty E_n$ together with the locally convex vector topology in which a convex subset $U$ is a $0$-neighborhood if and only if $U\cap E_n$ is a $0$-neighborhood in $E_n$, for each $n\in\N$.

The locally convex direct limit topology satisfies the following universal property: A linear map $\smfunc{f}{E}{F}$  to a locally convex space $F$ is continuous if and only if every restriction $\smfunc{f|_{E_n} }{E_n}{F}$ is continuous with respect to the topology of $E_n$.

Note: In general a locally convex direct limit of normed spaces need not be Hausdorff.
 In the examples of this paper we can show the Hausdorff property directly by constructing an injective continuous map into a suitable Hausdorff space.

\begin{proposition}[Characterization of zero neighborhoods]				\label{prop_direct_neighborhood}
 Let $\seqn{\delta_n}$ be a sequence of positive real numbers. Then the set
\[	\Vdelta:= \bigcup_{n\in \N} \left(\oBallin{\delta_1}{E_1}{0}+\cdots+\oBallin{\delta_n}{E_n}{0}\right)
 \]
is a $0$-neighborhood of the locally convex direct limit $E:= \bigcup_{n=1}^\infty E_n$. Furthermore, the sets of this type form a basis of $0$-neighborhoods.
\end{proposition}
\begin{proof}
 This is a well-known consequence of the fact that $E$ is a quotient of the direct sum $\bigoplus_n E_n$, equipped with the box topology.
\end{proof}

\begin{proposition}[Compact regularity]							\label{prop_compact_regularity}
 Let $E:= \bigcup_{n=1}^\infty E_n$ be a direct limit of Banach spaces. We assume $E$ is Hausdorff.
 Consider the following statements:
 \begin{itemize}
  \item [(i)]	For every $n\in\N$, there is a $0$-neighborhood $\Omega\subseteq E_n$ and a number $m\ge n$ and such that all the spaces $E_m, E_{m+1}, E_{m+2}, \ldots$ induce the same topology on $\Omega$.
  \item [(ii)]	The sequence $\seqn{E_n}$ is \emph{compactly regular}, i.e. every compact subset in $E$ is also a compact set in some $E_n$.
  \item [(iii)]	The locally convex vector space $E$ is complete.
 \end{itemize}
 Then (i) implies (ii) and (ii) implies (iii).
\end{proposition}
\begin{proof}
 This is the statement of \cite[Theorem 6.4 and corresp. Corollary]{MR1977923}.
\end{proof}
This proposition is relevant for two reasons: First, it is important to know whether the modelling space of a Lie group is complete. Second, we will show in a later work (see \cite{gloecknerdahmen}) that the Lie groups constructed in this paper are regular Lie groups in Milnor's sense and to do so it turns out to be important that they are modelled on a compactly regular direct limit.

\section{Analytic maps on LB-spaces}							\label{sec_LB}	
Now, we prove the main theorem of this paper, namely Theorem~A, stated in the introduction:

\begin{proof}[Proof of Theorem~A]
We start with some simplifications:
Since $f$ restricted to $\oBallin{R}{E_n}{0}$ is analytic and the intersection of each one dimensional affine subspace with $U$ is locally contained in 
some $\oBallin{R}{E_n}{0}$, the function $f$ is clearly Gateaux analytic.
By Lemma \ref{lemma_gateaux} it remains to show the continuity of $f$. The range space $F$ is locally convex and can therefore be embedded in a product of Banach spaces. We now use that a function into a product is continuous if and only if the projection onto each factor is continuous.
Therefore we can assume without loss of generality that $F$ is a Banach space.
Let $p\in U$ be given. It remains to show continuity of $f$ at $p$. Since $p\in\bigcup_{n\in \N} \oBallin{R}{E_n}{0}$, there is an index $m$ such that $p\in \oBallin{R}{E_m}{0}$. We may assume that $m=1$ since omitting only a finite number of spaces does not change the direct limit. Since $\oBallin{R}{E_1}{0}$ is open, there is an open ball $\oBallin{R'}{E_1}{p}\subseteq\oBallin{R}{E_1}{0}$. Then $\oBallin{R'}{E_n}{p}\subseteq \oBallin{R}{E_n}{0}$ for all $n\in\N$, using that $\opnorm{i_n}\leq 1$.
Now, we may restrict $f$ to this smaller subset and without loss of generality, we may assume that $R'=R$ and $p=0$. Therefore we only have to show continuity of $f$ at $0$. Furthermore, we may assume that $f(0)=0$ which can be obtained by a translation in $F$ which is clearly continuous. Let $r$ be a fixed positive real number strictly less that $\frac{R}{2e}$.
\newcommand{\bigsuminfty}{\sum_{\substack{\vec{j}\in\N^k}}}
\newcommand{\bigsumleqn}{\sum_{\substack{\vec{j}\in\N^k\\ \vec{j}\leq n}}}
\newcommand{\bigsummaxn}{\!\!\sum_{\substack{\vec{j}\in\N^k\\ \max \vec{j}= n}}}
\newcommand{\bigsumleqm}{\sum_{\substack{\vec{j}\in\N^k\\ \vec{j}\leq m}}}
\newcommand{\bigsummaxm}{\!\!\sum_{\substack{\vec{j}\in\N^k\\ \max \vec{j}= m}}}

We know that each $\func{f_n}{\oBallin{R}{E_n}{0}}{F}{x}{\sum_{k=1}^\infty \beta_{k,n}(x,\cdots,x)}$ is $C^\omega$ and bounded between normed complex vector spaces. Then by Lemma \ref{lemma_absolute}, we have the estimate:
\[
	\sum_{k=1}^\infty \opnorm{\beta_{k,n}} r^k \leq \frac{R}{R-2er}\maxnorm{f_n} =:S_n.		 
\]
Now, let $a_j:= \frac{r}{2^j}$ which implies that $r=\sum_{j=1}^\infty a_n$. Then we have for every $n\in\N$:
\begin{align*}
	S_n 	&\geq	\sum_{k=1}^\infty \opnorm{\beta_{k,n}} r^k	
		 =	\sum_{k=1}^\infty \opnorm{\beta_{k,n}} \left(\sum_{j=1}^\infty a_j\right)^k	
	\\&	 =	\sum_{k=1}^\infty \opnorm{\beta_{k,n}}\! \bigsuminfty a_{j_1} a_{j_2} \cdots a_{j_k}	\tag{$*$}\label{eqn_estimate}
\end{align*}
Let $\epsilon>0$ be given. We set $b_n := \min\left(1,\frac{\epsilon}{2^n\cdot S_n}\right)$ and $\delta_n := a_n\cdot b_n$. By construction it is clear that $\delta_n\leq a_n$ which will be used later.

To show continuity of $f$ at $0$, it suffices to show that the $0$-neighborhood $\Vdelta\subseteq E$ as defined in Proposition \ref{prop_direct_neighborhood} is a subset of the domain of $f$ and is mapped by $f$ into $\oBallin{\epsilon}{F}{0}$.

Let $x\in\Vdelta$. This means that there is a number $m\in\N$ such that $x=\sum_{j=1}^m x_j$ with $\norm{x_j}_{E_j}<\delta_j$.
We can estimate the $E_m$-norm of $x$ by 
\begin{align*}
			\biggnorm{\sum_{j=1}^m x_j}_{E_m}
		\leq 	\sum_{j=1}^m \norm{x_j}_{E_m} 
		\leq 	\sum_{j=1}^m \norm{x_j}_{E_j}
		< 	\sum_{j=1}^m \delta_j
		\leq 	\sum_{j=1}^m a_j
		<	r<R.
\end{align*}
So, $x\in \oBallin{R}{E_m}{0}\subseteq\bigcup_{n\in \N} \oBallin{R}{E_n}{0}$ which is the domain of $f$. So it makes sense to evaluate $f(x)$:
\begin{align*}
       f(x)   	&	=	      f_m(x)    
			= 	      \sum_{k=1}^\infty \beta_{k,m}\left(x,\ldots, x\right)   		
		   	=	      \sum_{k=1}^\infty \beta_{k,m}\left(\sum_{j_1=1}^m x_{j_1},\ldots, \sum_{j_k=1}^m x_{j_k}\right)   		
		\\&	=	      \sum_{k=1}^\infty \bigsumleqm \beta_{k,m}\left(x_{j_1},\ldots, x_{j_k}\right)   		
		   	=	      \sum_{k=1}^\infty \sum_{n=1}^m \bigsummaxn \beta_{k,n}\left(x_{j_1},\ldots, x_{j_k}\right)   		
\end{align*}
Note, that we used in the last line that $\beta_{k,m}$ and $\beta_{k,n}$ coincide when the arguments are elements of $E_n$.
Now, we can estimate the norm:
\begin{align*}
       \norm{f(x)}_F   	&	=	\Biggnorm{\sum_{k=1}^\infty \sum_{n=1}^m  \bigsummaxn \beta_{k,n}\left(x_{j_1},\ldots, x_{j_k}\right)}_F
			\\&	\leq	\sum_{k=1}^\infty \sum_{n=1}^m  \bigsummaxn \opnorm{\beta_{k,n}} \norm{x_{j_1}}_{E_n}\cdots \norm{x_{j_k}}_{E_n}
			\\&	\leq	\sum_{n=1}^m\sum_{k=1}^\infty \bigsummaxn \opnorm{\beta_{k,n}} \norm{x_{j_1}}_{E_{j_1}}\cdots \norm{x_{j_k}}_{E_{j_k}}
			\\&	\leq	\sum_{n=1}^m\sum_{k=1}^\infty \bigsummaxn \opnorm{\beta_{k,n}} \delta_{j_1}\cdots \delta_{j_k}
\end{align*}
One of the factors $\delta_{j_1}\cdots \delta_{j_k}$ is equal to $\delta_n= a_n\cdot b_n$, all the others can be estimated by the corresponding $a_j$:
\begin{align*}
       \norm{f(x)}_F   	&	\leq	\sum_{n=1}^m\sum_{k=1}^\infty \bigsummaxn \opnorm{\beta_{k,n}} a_{j_1}\cdots a_{j_k}\cdot b_n
			\\&	\leq	\sum_{n=1}^m b_n  \sum_{k=1}^\infty \opnorm{\beta_{k,n}}\!\!\bigsuminfty a_{j_1} a_{j_2} \cdots a_{j_k}
			   	\stackrel{\hbox{by (\ref{eqn_estimate}) } }{\leq}	\sum_{n=1}^m b_n \cdot S_n
			\\&	\leq	\sum_{n=1}^m \frac{\epsilon}{2^n\cdot S_n} \cdot S_n
			   	<   	\epsilon.
\end{align*}
This finishes the proof.
\end{proof}
Although this result explicitly needs that $\K=\C$, the following easy consequence also holds in the real case:

\begin{corollary}[Continuity of polynomials]						\label{cor_POLYNOMIAL}
 Let $\K\in\smset{\R,\C}$. A polynomial function defined on the direct limit $E=\bigcup_{n\in\N}E_n$ of normed $\K$-vector spaces $E_1\subseteq E_2 \subseteq \cdots$ with values in a locally convex space is continuous if and only if it is continuous on each step.
\end{corollary}
\begin{proof}
 First consider the case $\K=\C$. Let $\smfunc{f}{E}{F}$ be an $F$-valued polynomial map, defined on the locally convex direct limit $E=\bigcup_{n\in\N}E_n$ of normed spaces $E_n$. It is possible to choose on each $E_n$ an equivalent norm $\norm{\cdot}_{E_n}$ such that the continuous bonding maps 
 $\func{i_n }{E_n}{E_{n+1}}{x}{x}$ have operator norm at most  $1$. Set $R:=1$ and $U:=\bigcup_{n\in \N} \oBallin{R}{E_n}{0}$.
By hypothesis, each $f_n := f|_{\oBallin{R}{E_n}{0}}$ is continuous. A continuous polynomial is automaticly $C^\omega$ and locally bounded, i.e. it maps bounded sets to bounded sets, in particular,  $f_n(\oBallin{R}{E_n}{0})$ is bounded in $F$. Therefore we can directly apply Theorem~A 
and obtain that $f$ is analytic and hence continuous on the $0$-neighborhood $U$.
But for polynomial functions this is enough to guarantee continuity on the whole domain $E$. (see \cite[Theorem 1]{MR0313810})

Let $\K=\R$ now and let $E_\C$, $(E_n)_\C$ and $F_\C$ denote the complexifications of the $\R$-vector spaces $E$, $E_n$ and $F$ respectively. Then every polynomial map $\smfunc{f_n}{E_n}{F}$ can be extended to a complex polynomial map $\smfunc{(f_n)_\C}{E_\C}{F_\C}$. The maps $(f_n)_\C$ are continuous because the maps $f_n$ are so. We now apply the complex case and obtain that $f_\C$ is continuous and hence $f$ is continuous, too.
\end{proof}

\section[Example: Germs of diffeomorphisms]{Example: Germs of diffeomorphisms around a compact set in a Banach space}	\label{sec_DIFFEO}
Throughout this section, let $X$ be a Banach space over $\K$, and let $K\subseteq X$ be a nonempty compact subset.
We are interested in germs of analytic diffeomorphisms around $K$, i.e. we examine  analytic diffeomorphisms $\smfunc{\eta}{U_\eta}{V_\eta}$
where $U_\eta$ and $V_\eta$ are open subsets of $X$, both containing $K$, such that $\eta|_K=\id_K$. We identify two diffeomorphisms if they coincide on an open set $W\subseteq X$, containing $K$. It is easily checked that these equivalence classes of diffeomorphisms form a group with respect to composition. In this section we will turn this group into a Lie group modelled on a compactly regular direct limit of Banach spaces.

We will follow the strategy of \cite{MR2310802} to first consider the case $\K=\C$ and reduce the real case to the complex case.
The topologies used in \cite{MR2310802} do not work when $X$ is infinite dimensional. However, once we have constructed the Lie group structure for $\K=\C$, the proof of the real case can be copied verbatim from \cite[Corollary 15.11]{MR2310802}.

Therefore, from now on,  $\K = \C$.

\subsection*{The modelling space}
\newcommand{\BC}[3]{\ensuremath{\mathrm{BC}^{#3}(#1, #2)}}
\newcommand{\BCo}[3]{ \BC{#1}{#2}{\partial, #3} }
\newcommand{\BCzero}[3]{ \BC{#1}{#2}{#3}_{0}}
\newcommand{\compBC}[2]{\ensuremath{       g_{{\mathrm{BC} }, #1}^{#2}}}
\newcommand{\compNEU}[3]{\ensuremath{       g_{#1}^{#2,#3}}}
\newcommand{\compNEUNEU}[2]{\ensuremath{       g_{#1}^{#2}}}
\newcommand{\comp}[3]{\ensuremath{       g_{{#1}, #2}^{#3}}}
\newcommand{\dA}[4]{
\ensuremath{
d\ifthenelse{ \equal{#1}{} }{}{^{(#1)}}
       #2 \ifthenelse{\equal{#3}{} \and \equal{#4}{} }{}{(#3;#4)}
       }
}

Throughout this section, we will fix the following countable basis of open neighborhoods of $K$:
\[
 U_n := K + \oBallin{\frac{1}{n} }{X}{0} = \bigcup_{a\in K}\oBallin{\frac{1}{n} }{X}{a}.
\]
For every $n\in \N$, we define the following spaces:
\[
 \BC{U_n}{X}{0} := \set{\smfunc{\gamma}{U_n}{X} }{\gamma \hbox{ is continuous and bounded }}
\]
\[
 \HolbK{n} := \set{\smfunc{\gamma}{U_n}{X} }{\gamma \hbox{ is $C^\omega$, bounded and }\gamma|_K=0 }.
\]
It is well-known that $\BC{U_n}{X}{0}$ is a Banach space when equipped with the $\sup$-norm. The space $\HolbK{n}$ is a closed vector subspace of $\BC{U_n}{X}{0}$ and hence becomes a Banach space as well.

For every $k\in\N$ we define
\[
 \BC{U_n}{X}{k}:= \set{\smfunc{\gamma}{U_n}{X} }{\begin{array}{c}
 								 \gamma \hbox{ is $C^\omega$, bounded and the}\\
								 \hbox{ first $k$ \Frechet{} derivates are bounded}

\end{array} }.
\]
This space becomes a Banach space when endowed with the (finite number of) semi norms 
$\left( \gamma\mapsto \supnorm{\gamma^{(l)}} \right)$
where $l\in\smset{0,1,\ldots,k}$ and
$\smfunc{\gamma^{(k)} }{U_n}{\Lin^k(X,X)}$ denotes the $k$-th \Frechet{} derivative of $\gamma$.

The notation $\BC{U_n}{X}{k}$ suggests that the elements need not be analytic but only $k$ times continuously differentiable. However a map between complex Banach spaces which is $C_\C^1$ is automaticly Gateaux analytic (see \cite[Theorem 3.1]{MR0313811}) and continuous and therefore $C_\C^\omega$ by Lemma \ref{lemma_gateaux}.
Therefore the exponent $k$ only refers to the boundedness of the first $k$ derivatives.

\begin{lemma}									\label{lemma_inclusion}
 Let $n\in\N$ and $k\in\smset{0,1,2,\ldots}$. Then the linear operator $\nnfunc{\HolbK{n}}{\BC{U_{n+1} }{X}{k}}{\gamma}{\gamma|_{U_{n+1} } }$ is continuous.
\end{lemma}
\begin{proof}
 Let $x\in U_{n+1}$ be given. Then there is an $a\in K$ such that $x\in \oBallin{\frac{1}{n+1}}{X}{a}$. Set $R:=\frac{1}{n}-\frac{1}{n+1}=\frac{1}{n(n+1)}$, then $\oBallin{R}{X}{x}\subseteq \oBallin{\frac{1}{n}}{X}{a}\subseteq U_{n}$.
For each $\gamma \in \HolbK{n}$ we obtain a bounded analytic function $\smfunc{\gamma|_{\oBallin{R}{X}{x}} }{\oBallin{R}{X}{x}}{X}$.
We fix a real number $r<\frac{R}{2e}$ and apply Lemma \ref{lemma_absolute} to get the following estimate:
\[
	  \sum_{l=0}^\infty \frac{\opnorm{\gamma^{(l)}(x) } }{l!} r^l \leq \frac{R}{R-2er} \cdot \maxnorm{\gamma}.
\]

In particular we can estimate every summand in the infinite sum by the whole sum and conclude
\[
 	\opnorm{\gamma^{(l)}(x) } \leq \frac{l!}{r^l}\cdot \frac{R}{R-2er} \cdot \maxnorm{\gamma}.
\]
This bound does not depend on the choice of $x\in U_{n+1}$, hence
\[
 	\supnorm{ \gamma^{(l)} |_{U_{n+1}}}=\sup_{x\in U_{n+1}} \opnorm{\gamma^{(l)}(x) }\leq \frac{l!}{r^l}\cdot \frac{R}{R-2er} \cdot \maxnorm{\gamma}.\qedhere
\]
\end{proof}
We also need the space
\[
 \BConeK{n} := \set{\gamma\in\BC{U_{n} }{X}{1} }{ \gamma|_K=0}.
\]
This is a closed subspace of $\BC{U_{n} }{X}{1}$ and therefore becomes a Banach space with the induced topology.
\begin{lemma}									\label{lemma_Dnorm}
 The topology of the Banach space $\BConeK{n}$ is given by the norm
\[
 \Dnorm{\gamma}:=\supnorm{\gamma'}=\sup_{x\in U_n} \opnorm{\gamma'(x)}.
\]
\end{lemma}
\begin{proof}
 By definition, $\Dnorm{\cdot}$ is one of the continuous seminorms generating the topology of $\BConeK{n}$. It sufficed to show that $\supnorm{\cdot}$ is continuous with respect to $\Dnorm{\cdot}$. Then the  seminorm $\Dnorm{\cdot}$ generates a Banach space topology and therefore has to be a norm.

 Let $\gamma\in\BConeK{n}$ and let $x\in U_n = K+ \oBallin{\frac{1}{n} }{X}{0}$.
 Then $x=a+v$ with $a\in K$ and $\Xnorm{v}<\frac{1}{n}$.
 Then 
\begin{align*}
 \Xnorm{\gamma(x)} 	&= 	\Xnorm{\gamma(a+v)}
			= 	\XBiggnorm{\underbrace{\gamma(a)}_{=0} + \int_0^1 \gamma'(a+tv)(v) dt} 
		\\&	\leq	\int_0^1 \opnorm{\gamma'(a+tv)}\Xnorm{v} dt
		   	\leq	\Dnorm{\gamma}\cdot \frac{1}{n}.
\end{align*}
Therefore $\supnorm{\gamma}\leq \frac{1}{n}\Dnorm{\gamma}$ and this finishes the proof.
\end{proof}
Note that this does not work without the assumption that $\gamma|_K=0$.

From now on, the space $\BConeK{n}$ is endowed with the norm $\Dnorm{\cdot}$. In the proof of Lemma \ref{lemma_Dnorm} we have seen that
\[
 \nnfunc{\BConeK{n}}{\HolbK{n}}{\gamma}{\gamma}
\]
is a bounded operator of norm at most $\frac{1}{n}$.

We are now in the following situation: All of the arrows in the following diagram are injective bounded operators between Banach spaces:

\xymatrix{
 \HolbK{n}\ar[r]\ar[dr]		& \HolbK{n+1}\ar[r]\ar[dr]	& \HolbK{n+2}\ar[r]\ar[dr]	&\cdots				\\
 \BConeK{n}\ar[r]\ar[u]		& \BConeK{n+1}\ar[r]\ar[u]	& \BConeK{n+2}\ar[r]\ar[u]	& \cdots
}

\begin{proposition}
 The direct limit 
\[
 \GermK := \bigcup_{n\in\N} \HolbK{n} = \bigcup_{n\in\N} \BConeK{n}
\]
 is Hausdorff and compactly regular.
\end{proposition}
\begin{proof}
 To simplify notation, let $E_n := \HolbK{n}$.
 To see that the direct limit is Hausdorff, note that every $\gamma\in \HolbK{n}$ is uniquely determined if we know its power series expansion at each $a\in K$, since $U_n= \bigcup_{a\in K}\oBallin{\frac{1}{n} }{X}{a}$. Therefore the following mappings are injective:
 \[
  \func{\Phi_n}{\HolbK{n} }{\prod_{     \substack{a\in K\\ k\in\N}     } \Lin_c^k(X,X)   }{\gamma}{\left(  \gamma^{(k)}(a)      \right)_{a\in K,k\in \N} }
 \]
 Here, the $k$-linear continuous map $\smfunc{\gamma^{(k)}(a)}{X^k}{X}$ is the $k$-th \Frechet{} derivative and $\Lin_c^k(X,X)$ denotes the Banach space of continuous $k$-linear maps from $X^k$ to $X$. By Lemma \ref{lemma_inclusion}, calculating these derivatives is continuous with respect to the sup-norm, therefore the mappings above are continuous and linear. Since they are compatible with the bonding maps
$\nnfunc{\HolbK{n} }{\HolbK{n+1} }{\gamma}{\gamma|_{U_{n+1} } }$ we can extend these maps to the direct limit and obtain an injective continuous map into a Hausdorff space. This proves that the direct limit is Hausdorff.

To show compact regularity, we want to use Proposition \ref{prop_compact_regularity}. Therefore it suffices to show that
for every $n\in\N$, there is a $0$-neighborhood $\Omega\subseteq E_n$ and a number $m\ge n$ and such that all the spaces $E_m, E_{m+1}, E_{m+2},  \ldots$ 
induce the same topology on $\Omega$.
We need the following constant: $D:= \frac{3}{3-e}\approx 10,6489403$.
Let $n\in\N$ be given. Set $\Omega := \oBallin{1}{E_n}{0}$ and $m:=6n$.

Now we can apply Lemma \ref{lemma_absolute_family} with
 $E=F=X$ and $M=2\Omega=\oBallin{2}{E_n}{0}$, $R=\frac{1}{n}$ and $r=\frac{1}{6n}<\frac{R}{2e}$ and obtain
\[
	  \sum_{k=1}^\infty \underbrace{\sup_{\substack{\gamma\in 2\Omega\\a\in K} }  \frac{\opnorm{\gamma^{(k)}(a)}}{k!} }_{=:s_k}  \left(\frac{1}{6n}\right)^k
		\leq \frac{\frac{1}{n} }{\frac{1}{n} -2e\frac{1}{6n}} \cdot \sup_{\gamma\in 2\Omega} \supnorm{\gamma}= 2D.	\tag{$*$}	\label{eqn_star}
\]

Let $l\geq m$. To show that $E_l$ and $E_m$ induce the same topology on $\Omega$, it remains to prove that the inclusion map $\nnfunc{\Omega\subseteq E_l}{E_m}{\gamma}{\gamma}$ is continuous.
Let $\epsilon>0$ be given. By (\ref{eqn_star}) the series $\sum_{k=1}^\infty s_k \left(\frac{1}{6n}\right)^k$ converges.
Therefore there is a number $k_0\in\N$ such that $\sum_{k>k_0} s_k \left(\frac{1}{6n}\right)^k < \frac{\epsilon}{2}$.
We set $\delta := \frac{1}{D}\left( \frac{n}{l}  \right)^{k_0}\cdot \frac{\epsilon}{2}$. 

Now, let $\gamma_1,\gamma_2\in \Omega$ be two elements with $E_l$-distance $\norm{\gamma_1-\gamma_2}_{E_l}\leq \delta$. 
For the $E_m$-distance, we show that $\norm{\gamma_1-\gamma_2}_{E_m}\leq\epsilon$.
With $\gamma_d := \gamma_1-\gamma_2$, we know that $\gamma_d\in 2\Omega$.
It remains to show that $\norm{\gamma_d}_{E_m}=\sup_{x\in U_m} \Xnorm{\gamma_d(x)}\leq \epsilon$. Therefore let $x\in U_m$ be given.
By definition of $U_m = K + \oBallin{\frac{1}{m} }{X}{0}$, there is an $a\in K$ such that $x\in \oBallin{\frac{1}{m} }{X}{a}$. Now,
\begin{align*}
 \Xnorm{\gamma_d(x)}		&=	\Xnorm{	\sum_{k=1}^\infty \frac{\gamma_d^{(k)}(a)(x-a)  }{k!} 	} 
				\leq	\sum_{k=1}^\infty \frac{ \opnorm{\gamma_d^{(k)}(a)}  }{k!}\left(\frac{1}{m}\right)^k
			\\&	\leq	\sum_{k\leq k_0} \frac{ \opnorm{\gamma_d^{(k)}(a)}  }{k!}\left(\frac{1}{6n}\right)^k
				     +	\sum_{k> k_0} \frac{ \opnorm{\gamma_d^{(k)}(a)}  }{k!}\left(\frac{1}{6n}\right)^k
			\\&	\leq	\sum_{k\leq k_0} \frac{ \opnorm{\gamma_d^{(k)}(a)}  }{k!}\left(\frac{1}{6l}\right)^k\left(\frac{l}{n}\right)^{k_0}
				     +	\sum_{k> k_0} s_k \left(\frac{1}{6n}\right)^k
			\\&	\leq	\left(\frac{l}{n}\right)^{k_0}\sum_{k=1}^\infty \frac{ \opnorm{\gamma_d^{(k)}(a)}  }{k!}\left(\frac{1}{6l}\right)^k
				     +	\frac{\epsilon}{2}
			\\&	\leq	\left(\frac{l}{n}\right)^{k_0}\!\cdot D\cdot \underbrace{\norm{\gamma_d}_{E_l}}_{<\delta}
				     +	\frac{\epsilon}{2}
			   	\leq	\left(\frac{l}{n}\right)^{k_0}\!\cdot D\cdot \frac{1}{D}\left( \frac{n}{l}  \right)^{k_0}\!\cdot \frac{\epsilon}{2}
				     +	\frac{\epsilon}{2} =\epsilon.
\end{align*}
This shows that the space
\[
\GermK =\bigcup_{n\in\N} \HolbK{n}=\bigcup_{n\in\N} \BConeK{n}
\]
is a compactly regular LB-space.
\end{proof}

\subsection*{The monoid}
To turn $\DiffGermK$ into a Lie group modelled on $\GermK$, we first construct an analytic structure on the monoid
\[
 \EndGermK :=  \set{    \smfunc{\eta}{U_\eta}{X}  }{  \begin{array}{c}
 								 \eta \hbox{ is a $C^\omega$-map, $U_\eta$ is an open} \\
								\hbox{ neighborhood of $K$ and $\eta|_K=\id_K$} 

\end{array}
 }/_\sim
\]
 where $\eta_1 \sim \eta_2$ if and only if they coincide on a common neighborhood of $K$.

Since every neighborhood of $K$ contains one of the neighborhoods $U_n$ it suffices to look at analytic maps of the form $\eta=\id_{U_n} + \gamma$ with $\gamma\in \HolbK{n}$ for any $n\in\N$. This implies that the following map is a bijection:

\[
 \Func{\glchart}{\EndGermK}{\GermK}{ \cl{ \gamma + \id_{U_n} } }{\gamma\in \HolbK{n}}
\]

We use the map $\glchart$ as a global chart and define the manifold structure on $\EndGermK$ such that $\glchart$ is a diffeomorphism.

We will use the following result from \cite{ArtikelBoris}:

\begin{lemma}									\label{lemma_boris}
	       Let $X$, $Y$ and $Z$ be normed spaces over $\C$,
	       $U\subseteq X$ and $V\subseteq Y$ open subsets
		and $k \in \smset{0,1,2,\ldots}$.
	       Then the continuous map
	       \[
	               \compNEUNEU{Z}{k} :
	               \BC{V}{Z}{k + 2} \times \BCo{U}{V}{k}
	               \to \BC{U}{Z}{k}
	               : (\gamma,\eta)\mapsto \gamma\circ\eta
	       \]
	       is a $C^\omega$-map with \Frechet{} derivative
	       \begin{equation}\label{Ableitung_Comp_BCellxBC}
	               \left(\compNEUNEU{Z}{k}\right)'(\gamma_0,\eta_0)(\gamma,\eta)
	               = \compNEUNEU{Z}{k}(\gamma,\eta_0)
	               + \compNEUNEU{\Lin(Y,Z)}{k}(\gamma_0',\eta_0) \cdot \eta.
	       \end{equation}
\end{lemma}
Here, $\BCo{U}{V}{k}$ is the set of all $\gamma\in \BC{U}{Y}{k}$ whose image is contained in $V$ and has a positive distance to the boundary of $V$. It is open in $\BC{U}{Y}{k}$. 

\begin{proposition}									\label{prop_monoid_mult}
 The monoid multiplication of $\EndGermK$:
\[
 \Func{\mu}{\EndGermK\times\EndGermK}{\EndGermK}{(\cl{\eta_1} , \cl{\eta_2} )}{ \cl{\eta1 \circ \eta_2} }
\]
 is $C^\omega$ with respect to the manifold structure defined by $\glchart$.
\end{proposition}
\begin{proof}
Using the global chart $\glchart$, this map becomes
\begin{align*}
 \glchart \circ \mu \circ \left(\glchart\times \glchart\right)^{-1}: &   \GermK\times\GermK \longrightarrow \GermK	\\
			(\gamma_1,\gamma_2) & \mapsto (\gamma_1+\id)\circ(\gamma_2+\id)-\id = \gamma_1\circ(\gamma_2+\id)+\gamma_2.
\end{align*}
To show analyticity of that map, it suffices to show that 
\[
 \Func{f}{\GermK\times \GermK}{\GermK}{(\gamma_1 , \gamma_2)}{\gamma_1\circ(\gamma_2+\id)}
\]
is analytic.

For each $n\in\N$, we set
\begin{align*}
	E_n &:=	\left(\HolbK{n},\supnorm{\cdot}\right), \\
	F_n &:=	\left(\BConeK{n},\Dnorm{\cdot}\right).
\end{align*}
The domain of the map $f$ in question can now be regarded as the following direct limit: $\GermK\times\GermK = \bigcup_{n\in\N} E_n\times F_n$.
For each $R>0$ we set
$\Omega_R := \bigcup_{n\in\N}\oBallin{R}{F_n}{0}$.
One easily checks that
\begin{align*}
	\bigcup_{R>0} \Omega_R			&=	\GermK			\\
	\bigcup_{n\in\N}\oBallin{R}{E_n}{0}	&=	\GermK &\hbox{ for every $R>0$.} 
\end{align*}
Therefore the domain $\GermK\times\GermK$ can be written as a union $\bigcup_{R\in\N} \left( \GermK\times\Omega_R \right)$ of open $0$-neighborhoods. This means that $f$ is analytic on $\GermK\times\GermK$ if and only if $f$ is analytic on each $\GermK\times\Omega_R$.

Let $R>0$ be given, without loss of generality, $R\in\N$.
To simplify notation, we denote the restriction of $f$ to $\GermK\times\Omega_R$ also by $f$.

Now, define $l_n := (R+1)(n+2)\in \N$. Since $\lim_{n\rightarrow\infty}l_n=\infty$, the sequence $\seqn{F_{l_n}}$ is cofinal in $\seqn{F_{n}}$, hence
\[
 \GermK\times\Omega_R 	= \bigcup_{n\in\N}\left( \oBallin{R}{E_n}{0}\times\oBallin{R}{F_{l_n} }{0} \right)
			= \bigcup_{n\in\N} \oBallin{R}{H_n}{0}.
\]
Here, we set $H_n := E_n\times F_{l_n}$ with the norm
\[
 \norm{(\gamma_1,\gamma_2)}_{H_n}:=\max\smset{\maxnorm{\gamma_1},\Dnorm{\gamma_2}}=\max\smset{\maxnorm{\gamma_1},\maxnorm{\gamma_2'}}.
\]
All bonding maps $\smfunc{i_n}{H_n}{H_{n+1}}$ have operator norm at most $1$. We now would like to apply Theorem~A.
 To this end, we define
\[
 \func{f_n}{\oBallin{R}{H_n}{0} }{\GermK}{(\gamma_1,\gamma_2) }{\gamma_1\circ(\gamma_2+\id_{U_{l_n}})}.
\]
We claim:
\begin{itemize}
 \item [(a)] Each $f_n$ makes sense,
 \item [(b)] Each $f_n$ is $C^\omega$,
 \item [(c)] Each $f_n$ is bounded.
\end{itemize}
Once we have this, by Theorem~A
 the map $f$ is analytic, as we had to show.
	(a)
	Let $(\gamma_1,\gamma_2)\in \oBallin{R}{H_n}{0}=\oBallin{R}{E_n}{0}\times\oBallin{R}{F_{l_n} }{0}$. We have to show that $\gamma_1$ and $(\gamma_2+\id_{U_{l_n}})$ can be composed, i.e. that $(\gamma_2+\id_{U_{l_n}})(U_{l_n})\subseteq U_n$.
	In fact, we actually show that $(\gamma_2+\id_{U_{l_n}})(U_{l_n})\subseteq U_{n+2}$.

	Therefore, let $x\in U_{l_n}=K+\oBallin{\frac{1}{l_n} }{X}{0}$ be given. Then $x$ is of the form $x=a+v$ with $a\in K$ and $\Xnorm{v}<\frac{1}{l_n}$.
	Now, we apply $(\gamma_2+\id_{U_{l_n}})$ to $x$:
	\[
	 (\gamma_2+\id_{U_{l_n}})(x)=\gamma_2(a+v)+a+v = a+w
	\]
	with $w:=v+\gamma_2(a+v)$. Now, we estimate the norm of $w$:
	\begin{align*}
	 \Xnorm{w}	&	=	\Xnorm{v+\gamma_2(a+v)}
				\leq	\Xnorm{v}	+ \Xnorm{\gamma_2(a+v)}
			\\&	=	\Xnorm{v}	+ \Xnorm{\gamma_2(a+v)-\gamma_2(a)}
			\\&   	=	\Xnorm{v}	+ \Xnorm{  \int_0^1 \gamma_2'(a+tv)(v) dt  }
			\\&   	<	\frac{1}{l_n}	+ \sup_{t\in[0,1]}\opnorm{\gamma_2'(a+tv)} \Xnorm{v}
			      	\leq	\frac{1}{l_n}	+ \Dnorm{\gamma_2} \Xnorm{v}
			\\&   	\leq	\frac{1}{l_n}	+ R \frac{1}{l_n}
			   	=   	\frac{R+1}{(R+1)(n+2)} = \frac{1}{n+2}.
	\end{align*}
	Therefore $(\gamma_2+\id_{U_{l_n}})(x)\in \oBallin{\frac{1}{n+2} }{X}{a}\subseteq U_{n+2}$.
        (b) 
\newcommand{\OBEN}{\clubsuit}
\newcommand{\UNTEN}{\heartsuit}
\newcommand{\LINKS}{\diamondsuit}
\newcommand{\RECHTS}{\spadesuit}
	The image of $f_n$ is a subset of $\BConeK{l_n}$. The inclusion map: 
	\hbox{$\nnsmfunc{\BConeK{l_n} }{\GermK}$} is continuous linear and therefore $C^\omega$. It remains to show that the arrow $\OBEN$ in the following diagram is $C^\omega$:
\[
\!\!\!\!\!\!\!\!\!\!\!\!\!
	\xymatrix@C=10pt{ \oBallin{R }{H_n }{0 } \ar[rr]_-\OBEN \ar@/^1pc/[rrr]^{f_n}	\ar[d]^\LINKS
						&&	\BConeK{l_n}	\ar@{^{(}->}[d]_\RECHTS	\ar[r] & \GermK			\\
	\BC{U_{n+1} }{X}{3}\times\BCo{U_{l_n}}{U_{n+1}}{1} \ar[rr]^-{\UNTEN}
						&&	\BC{U_{l_n}}{X}{1}
 }
\]
	The space $\BConeK{l_n}$ is a closed subspace of $\BC{U_{l_n}}{X}{1}$ and $\RECHTS$ is a topological embedding.
	Therefore $f_n$ will be $C^\omega$ if we are able to show that $\RECHTS \circ \OBEN$ is so.

	Let $(\gamma_1,\gamma_2)\in \oBallin{R}{H_n}{0}$. Then $\gamma_1$ is complex analytic and bounded on $U_n$.
	We have seen in Lemma \ref{lemma_inclusion} that all derivatives of $\gamma_1$ are bounded when restricting $\gamma_1$ to the smaller set $U_{n+1}$ and that the inclusion $ \nnfunc{\HolbK{n} }{\BC{U_{n+1} }{X}{k} }{\gamma_1}{\gamma_1|_{U_{n+1} }}$ is continuous for every $k\in\N$.
	
	We have just shown in (a) that the image of $(\gamma_2+\id_{U_{l_n}})$ is a subset of $U_{n+2}$, hence it has a positive distance from the boundary of $U_{n+1}$. This means it lies in the space $\BCo{U_{l_n} }{U_{n+1} }{1}$ as defined in Lemma \ref{lemma_boris}. The map
	
	\[
	\func{\LINKS}{\oBallin{R}{H_n}{0} }{\BC{U_{n+1}}{X}{3}\times \BCo{U_{l_n} }{U_{n+1} }{1} }{ (\gamma_1,\gamma_2) }{ (\gamma_1,\gamma_2+\id_{U_{l_n}}) }
	\]
	in the diagram above is therefore well-defined and continuous. Since it is affine, it is automatically analytic.
	
	To make the diagram commutative, we define the remaining arrow as
	\[
	 \Func{\UNTEN}{\BC{U_{n+1}}{X}{3} \times \BCo{U_{l_n} }{U_{n+1}}{1} }{\BC{U_{l_n} }{X}{1} }{(\gamma,\eta) }{\gamma\circ\eta }
	\]
 	and this is $C^1$ by Lemma \ref{lemma_boris} (with $k=l=1$). Since we are dealing with mappings between complex Banach spaces, the $C^1$-property implies complex analyticity. 
       
	(c) 
	Let $(\gamma_1.\gamma_2)\in \oBallin{R}{E_n}{0}\times\oBallin{R}{F_{l_n}}{0}$. Then $f_n(\gamma_1,\gamma_2)=\gamma_1\circ(\gamma_2+\id_{U_{l_n}})$ is an element of $\HolbK{l_n}$ of norm $\maxnorm{\gamma_1\circ(\gamma_2+\id_{U_{l_n}}) }\leq \maxnorm{\gamma_1}<R$. Therefore the image of $f_n$ is a bounded subset of $\HolbK{l_n}$ and hence a bounded subset of the direct limit $\GermK$.

Therefore, by Theorem~A, 
 $f$ is complex analytic and we have shown that $\EndGermK$ is a complex analytic monoid.
\end{proof}

\subsection*{The group}
The monoid $\GermK$ has a $C^\omega$-manifold structure and an analytic multiplication. We now show that the group of invertible elements of the monoid is open and that inversion is analytic.

For the openness, we use a lemma:

\begin{lemma}									\label{lemma_157}
 Let $\gamma\in \BConeK{n}$ with $\Dnorm{\gamma}=\supnorm{\gamma'}<1$. Then $\eta := \id_{U_{6n}}+\gamma|_{U_{6n}}$ is a $C^\omega$-diffeomorphism onto its open image.
\end{lemma}
\begin{proof}
 Let $x\in U_{6n}$. Then the \Frechet{} derivative $\eta'(x)$ at $x$ is an element in the Banach algebra $\BoundOpNorm{X}$. The distance between $\eta'(x)$ and the identity of the algebra is $\opnorm{\eta'(x) - \id_X}= \opnorm{\gamma'(x)}\leq \maxnorm{\gamma'}<1$. Therefore $\eta'(x)\in \oBallin{1}{\BoundOp{X} }{\id_X}\subseteq \GL{X}$. By the Inverse Function Theorem for complex Banach spaces this implies that there is an open neighborhood of $x$ on which $\eta$ is a diffeomorphism onto its open image. Since $x\in U_{6n}$ was arbitrary, we know that the image $\eta(U_{6n})$ is open. To show that $\smfunc{\eta}{U_{6n}}{\eta(U_{6n})}$ is not only a local, but a global diffeomorphism, it remains to show injectivity of $\eta$.

 Let $x,y\in U_{6n}$ with $\eta(x)=\eta(y)$ be given. We have to show that $x=y$.
 This is easy once we have shown that the line segment joining $x$ and $y$ lies in $U_n$.
 By definition of $U_{6n}$, there are elements $a,b\in K$ and $v,w\in X$ such that $\Xnorm{v},\Xnorm{w}<\frac{1}{6n}$ and $x=a+v, y= b+w$.
 Let $[a,x]:=\set{a+tv}{t\in[0,1]} \subseteq \oBallin{\frac{1}{6n} }{X}{a} \subseteq U_{6n}$ denote the compact line segment joining $a$ and $x$. Then
 \begin{align*}
 	\Xnorm{\eta(x)-\eta(a)} &	=	\Xnorm{\int_0^1 \eta'(a+tv)(v)dt}
					\leq	\max_{t\in[0,1]} \Xnorm{\eta'(a+tv)} \cdot \Xnorm{v}
				\\&	\leq	 \underbrace{ \left(  \opnorm{\id}+\opnorm{\gamma'}  \right)}_{<2}\cdot \Xnorm{v}
					<	2\cdot \frac{1}{6n}=\frac{1}{3n}.
 \end{align*}
 Likewise we see that $\Xnorm{\eta(y)-\eta(b)} <\frac{1}{3n}$.
 We can now estimate the distance between the points $a$ and $b$:
 \begin{align*}
  	\Xnorm{a-b}		&	=	\Xnorm{\eta(a)-\eta(b)}
			\\&		\leq	\underbrace{\Xnorm{\eta(a)-\eta(x)} }_{<\frac{1}{3n} }
					+   	\underbrace{\Xnorm{\eta(x)-\eta(y)} }_{=0}
					+   	\underbrace{\Xnorm{\eta(y)-\eta(b)} }_{<\frac{1}{3n} }
					<	\frac{2}{3n}.
 \end{align*}
 This also allows us to estimate the distance between $y$ and $a$:
 \begin{align*}
  	\Xnorm{y-a}		&	\leq	\Xnorm{y-b}+\Xnorm{b-a} < \frac{1}{6n} + \frac{2}{3n}<\frac{1}{n}.
 \end{align*}
 So, $y\in \oBallin{\frac{1}{n} }{X}{a}$. Therefore the two points $x$ and $y$ both lie in the convex set $\oBallin{\frac{1}{n} }{X}{a}$.
 Therefore also the line segment $[x,y]$ lies in $\oBallin{\frac{1}{n} }{X}{a}$ which is a subset of $U_n$, and thus
\begin{align*}
 	0			&	=	\Xnorm{ \eta(x)-\eta(y) }
					=   	\Xnorm{ x-y + \gamma(x) - \gamma(y) }
				\\&	\geq	\Xnorm{x-y} - \Xnorm{\gamma(x) - \gamma(y) }
				\\&   	=	\Xnorm{x-y} - \Xnorm{\int_0^1 \gamma'(y+t(x-y))(x-y) dt}
				\\&	\geq	\Xnorm{x-y} - \supnorm{\gamma'}\Xnorm{x-y}
					=	\Xnorm{x-y} \underbrace{(1 - \supnorm{\gamma'})}_{>0}.
\end{align*}
 Therefore $\Xnorm{x-y}$ has to be zero and so $\smfunc{\eta}{U_{6n}}{X}$ is injective. This finishes the proof.
\end{proof}

\begin{proposition}
 Let $\GermEndK^\times$ denote the group of invertible elements of $\GermEndK$ and let
\[
 \DiffGermK :=  \set{    \eta  }{  \begin{array}{c}
 								 \eta \hbox{ is a $C^\omega$-diffeomorphism between open} \\
								\hbox{ neighborhoods of $K$ and $\eta|_K=\id_K$} 

\end{array}
 }/_\sim,
\]
 where two diffeomorphisms $\eta_1 \sim \eta_2$ are identified if they coincide on a common neighborhood of $K$.
 Then $\DiffGermK=\GermEndK^\times$ and this is an open subset of $\GermEndK$.
\end{proposition}
\begin{proof}
 If $\eta_1$ is a diffeomorphism between open neighborhoods of $K$, then $\cl{\eta_1}$ is clearly invertible and thus $\GermDiffK\subseteq \GermEndK^\times$.
 If $\cl{\eta_1},\cl{\eta_2}\in \GermEndK^\times$ with $\cl{\eta_1}\circ\cl{\eta_2}=\cl{\id_X}$, then $\eta_1\circ \eta_2|_W = \id_W$ on some open neighborhood $W$ of $K$, whence $\eta_2|_W$ is injective and $\eta_2'(W)\subseteq \GL{X}$. By the Inverse Function Theorem for Banach spaces this implies that $\eta_2$ is a diffeomorphism onto an open neighborhood of $K$ and thus $\GermDiffK \supseteq \GermEndK^\times$.

 The set $U:= \bigcup_{n\in\N}\oBallin{1}{\BConeK{n} }{0}\subseteq  \GermK$ is an open $0$-neigh\-bor\-hood in $\GermK$. Using the global chart, we see that $\glchart^{-1}(U)$ is an open $\cl{\id}$-neighborhood in $\EndGermK$. By Lemma \ref{lemma_157} we know that every $\gamma\in \glchart^{-1}(U)$ is a diffeomorphism onto an open image and thus $\glchart^{-1}(U)\subseteq \DiffGermK=\GermEndK^\times$. 
 Therefore the unit group of the monoid contains an open identity neighborhood, and hence the whole unit group has to be open.
\end{proof}

From Lemma \ref{lemma_157}, we know that the image of $\smfunc{\eta}{U_{6n}}{X}$ is an open neighborhood of $K$ and therefore has to contain one of the basic neighborhoods $U_m$ for an $m\in \N$. The next lemma provides quantitative information:
\begin{lemma}	\label{lemma_12n}
 Let $\gamma\in \BConeK{n}$ with $\Dnorm{\gamma}=\supnorm{\gamma'}\leq \frac{1}{2}$ and let $\eta := \id_{U_{6n}}+\gamma|_{U_{6n}}$ be as in Lemma \ref{lemma_157}. Then the image of $\eta$ contains $U_{12n}$ and we have
		\[
	\supnorm{  \left.  \left(    \left(  \gamma + \id_{U_n} \right){\Big|}_{U_{6n} }  \right)^{-1}\right|_{U_{12n} }  -\id_{U_{12n}}    }\leq \frac{1}{6n}.
		\]
\end{lemma}

To prove this lemma, we need a quantitative version of the Inverse Function Theorem for Banach spaces which can be found in \cite{MR0418164}:
\begin{theorem}[Lipschitz inverse function theorem]	\label{thm_Lipschitz_Inverse}
 Let $X$ be a Banach space and let $\smfunc{T}{X}{X}$ be a linear invertible map. Suppose $\smfunc{f}{U}{X}$ is $L$-Lipschitz continuous with $L>0$, where $U$ an open neighborhood of $0$ in $X$ and $f(0)=0$, and $\lambda:=L\cdot \opnorm{T^{-1}}<1$. Then $T+f$ is a homeomorphism of $U$ onto an open subset $V$ of $X$ and $(T+f)^{-1}$ is Lipschitz with constant \smash{$\frac{1}{1-\lambda}\opnorm{T^{-1}}$}. If $U$ contains the ball $\oBallin{r}{X}{0}$, then $V$ contains the ball $\oBallin{r'}{X}{0}$ with $r':=\frac{r(1-\lambda)}{\opnorm{T^{-1}} }$.
\end{theorem}

\begin{proof}[Proof of Lemma \ref{lemma_12n}:]
 Let \smash{$x\in U_{12n}=K+\oBallin{\frac{1}{12n} }{X}{0}$} be given. We have to show that $x\in \eta(U_{6n})$.
 We know that there is an $a\in K$ such that $x=a+v$ with $v\in \oBallin{\frac{1}{12n} }{X}{0}$. Now, we set
	$r:=\frac{1}{6n}$, $T:=\id_X$, $U:=\oBallin{\frac{1}{6n}}{X}{0}$ and $\func{f}{U}{X}{w}{\gamma(a+w)}$. This function satisfies $f(0)=0$ and is Lipschitz continuous with Lipschitz constant $L:=\supnorm{f'}\leq \Dnorm{\gamma}\leq \frac{1}{2}$. The number $\lambda:= L\cdot \opnorm{T^{-1}}=L\leq \frac{1}{2}$ is strictly less than $1$ and therefore all hypotheses of Theorem~\ref{thm_Lipschitz_Inverse} are satisfied. Therefore we may conclude that the image of $(\id+f)$ contains the ball $\oBallin{r'}{X}{0}$ with $r'=\frac{r(1-\lambda)}{\opnorm{T^{-1}} }=\frac{1}{6n}\cdot(1-\lambda)\geq \frac{1}{6n}(1-\frac{1}{2})=\frac{1}{12n}$. So, there exists a $w\in U$ such that $(\id+f)(w)=v$. But this means:
\[
 x=a+v=a+(\id+f)(w)=a+w+f(w)=a+w+\gamma(a+w)=\eta(a+w)
\]
So $x$ is in the image of $\eta$. This proves $U_{12n}\subseteq \eta(U_{6n})$.

 Since the \Frechet{} derivative of $\eta := \id_{U_{6n}}+\gamma|_{U_{6n}}$ has distance at most $\frac{1}{2}$ from the identity, it follows that the \Frechet{} derivative of $\eta^{-1}$ has distance at most $\frac{1}{1-\frac{1}{2}}=2$ from the identity. Therefore:
 \[
	\Dnorm{  \left.  \left(    \left(  \gamma + \id_{U_n} \right){\Big|}_{U_{6n} }  \right)^{-1}\right|_{U_{12n} }  -\id_{U_{12n}}    }\leq 2.
 \]
 Together with $\supnorm{\cdot}\leq \frac{1}{12n}\Dnorm{\cdot}$ the assertion follows.
\end{proof}

So far we showed that $\DiffGermK$ is an open subset of the $C^\omega_\C$-manifold $\EndGermK$ and therefore has an induced manifold structure.
To show complex analyticity of the inversion map, we once again use our global chart $\glchart$ and define:
\[
 \Func{\Inversion}{ \glchart(\DiffGermK) }{\glchart(\DiffGermK)}{\gamma}{ \glchart\left(      \left(\glchart^{-1}(\gamma)\right)^{-1}   \right). }
\]
It remains to show that $\Inversion$ is analytic.

From now on, we again use the notation: $E_n := \HolbK{n}$ and $F_n := \BConeK{n}$.
Lemma \ref{lemma_12n} allows us to define for every $n\in\N$ the following map:
\[
 \Func{\Inversion_n}{\oBallin{\frac{1}{2}}{F_n}{0} }{\oBallin{\frac{1}{6n} }{E_{12n} }{0}  }
			{ \gamma }{  \left.  \left(    \left(  \gamma + \id_{U_n} \right){\Big|}_{U_{6n} }  \right)^{-1}\right|_{U_{12n} }  -\id_{U_{12n}}   . }
\]
If we are able to show that every $\Inversion_n$ is $C^\omega$ then we can directly apply Theorem~A
 and see that the monoid inversion is analytic on an open neighborhood of the identity. Then inversion is everywhere $C^\omega$ and we are done.

\newcommand{\direcEins}{\widehat{\gamma}_1}
\newcommand{\direc}{\widehat{\gamma}_2}

\begin{proposition}										\label{proposition_h_n}
\begin{itemize}
 \item [(a)] The mapping
	\[
	 \Func{h_n}{\oBallin{\frac{1}{2}}{F_n}{0} \times \oBallin{\frac{1}{6n} }{E_{12n} }{0}   }{E_{12n} }{(\gamma_1,\gamma_2)}
									{(\gamma_1+\id_{U_n})\circ (\gamma_2+\id_{U_{12n}}) - \id_{U_{12n}}	}
	\]
	is complex analytic.
 \item [(b)]	For every fixed $(\gamma_1,\gamma_2)\in \oBallin{\frac{1}{2}}{F_n }{0} \times \oBallin{\frac{1}{6n} }{E_{12n} }{0}$ and every $\direcEins\in F_n, \direc\in E_{12n}$ we have
	\[
	 h_n'(\gamma_1,\gamma_2)(\direcEins,\direc) =
			 \direcEins\circ(\gamma_2+\id_{U_{12n}} )+ \gamma_1'(\gamma_2+\id_{U_{12n}} )(\direc)+\direc.
	\]
 \item [(c)] For $(\gamma_1,\gamma_2)\in \oBallin{\frac{1}{2}}{F_{n} }{0} \times \oBallin{\frac{1}{6n}}{E_{12n} }{0} $, we have the equivalence:
	\[
 	 \bigl(h_n(\gamma_1,\gamma_2)=0 \bigr) \Longleftrightarrow   \bigl(\gamma_2 = \Inversion_n(\gamma_1)\bigr).
	\]
 \item [(d)] Every $\Inversion_n$ is complex analytic.
\end{itemize}
\end{proposition}
\begin{proof}
\newcommand{\OBEN}{\clubsuit}
\newcommand{\UNTEN}{\heartsuit}
\newcommand{\LINKS}{\diamondsuit}
\newcommand{\RECHTS}{\spadesuit}
        (a)
	The argument is essentially the same as in Proposition \ref{prop_monoid_mult}. We write
	$ h_n(\gamma_1,\gamma_2) = \OBEN(\gamma_1,\gamma_2)+\gamma_2$ with $\OBEN(\gamma_1,\gamma_2)=\gamma_1\circ(\gamma_2+\id_{U_{12n} })$
	and have the following commutative diagram:
	\[
	 \xymatrix{ \oBallin{\frac{1}{2}}{F_n }{0} \times \oBallin{\frac{1}{6n} }{E_{12n} }{0}  \ar[rr]^-\OBEN 	\ar[d]^\LINKS
							&&	E_{12n}		\ar@{^{(}->}[d]_\RECHTS				\\
		\BC{U_{2n} }{X}{2}\times\BCo{U_{12n}}{U_{2n}}{0} \ar[rr]^-\UNTEN
							&&	\BC{U_{12n}}{X}{0}
	 }
	\]
	Once again $\RECHTS$ is a topological embedding.
	The map 
	\[
	\Func{\LINKS}{\oBallin{\frac{1}{2}}{F_n }{0} \times \oBallin{\frac{1}{6n} }{E_{12n} }{0} }{\BC{U_{2n}}{X}{2}\times \BCo{U_{12n} }{U_{2n} }{0} }{ (\gamma_1,\gamma_2) }{ (\gamma_1|_{U_{2n} },\gamma_2+\id_{U_{12n}}) }
	\]
	is  well-defined and continuous. Since it is affine, it is automatically analytic. The last arrow
	\[
	 \Func{\UNTEN}{\BC{U_{2n}}{X}{2} \times \BCo{U_{12n} }{U_{2n}}{0} }{\BC{U_{12n} }{X}{0} }{(\gamma,\eta) }{\gamma\circ\eta }
	\]
	is $C^\omega$ by Lemma \ref{lemma_boris} and since the diagram commutes, $h_n$ is analytic.

	(b) 
	This follows directly from the formula in Lemma \ref{lemma_boris}.
        
	(c) 
	Assume that $\gamma_2 = \Inversion_n(\gamma_1)$ holds. 	Then
	\begin{align*}
	 h_n(\gamma_1, \gamma_2 ) 	&	=	(\gamma_1+\id_{U_n})\circ (\Inversion_n(\gamma_1)+\id_{U_{12n}}) - \id_{U_{12n}}	
					\\&	=	(\gamma_1+\id_{U_n})\circ    \left.  \left(    \left(  \gamma_1 + \id_{U_n} \right){\Big|}_{U_{6n} }  \right)^{-1}\right|_{U_{12n} }  \!\!\!\! - \id_{U_{12n}}	
					\\&	=	\id_{U_{12n}}   - \id_{U_{12n}}
						=	 0.
	\end{align*}
	Conversely, assume that $(\gamma_1,\gamma_2)\in \oBallin{\frac{1}{2}}{F_n }{0} \times \oBallin{\frac{1}{6n} }{E_{12n} }{0}$ is given with $h_n(\gamma_1,\gamma_2)=0$.
	Then $(\gamma_1+\id_{U_n})\circ(\gamma_2+\id_{U_{12n}})=\id_{U_{12n}}$.
	Since $\gamma_2+\id_{U_{12n}}$ is continuous, $W:= (\gamma_2+\id_{U_{12n}})^{-1}(U_{6n})\subseteq U_{12n}$ is an open $K$-neighborhood. Moreover,
	\[
	    \left(  \gamma_1 + \id_{U_n} \right){\big|}_{U_{6n} }\circ(\gamma_2+\id_{U_{12n}})|_W=\id_{W}.
	\]
	But since $\left(  \gamma_1 + \id_{U_n} \right){\big|}_{U_{6n} }$ is a diffeomorphism, we can compose this equality from the left with
	$\left(\left(  \gamma_1 + \id_{U_n} \right){\Big|}_{U_{6n} }\right)^{-1}$ and obtain
	\[
		(\gamma_2+\id_{U_{12n}})|_W=\left.  \left(    \left(  \gamma_1 + \id_{U_n} \right){\Big|}_{U_{6n} }  \right)^{-1}\right|_W.
	\]
	Thus we obtain that $\gamma_2$ and $\Inversion_n(\gamma_1)$ coincide on a smaller neighborhood $W\subseteq U_{12n}$ and since they are complex analytic this means $\gamma_2 = \Inversion_n(\gamma_1)$.
        
	(d) 
	Let $\gamma_1\in \oBallin{\frac{1}{2}}{F_n}{0}$ and set $\gamma_2:= \Inversion_n(\gamma_1)\in \oBallin{\frac{1}{6n}}{E_{12n} }{0}$. By (c) this implies $h_n(\gamma_1,\gamma_2)=0$.
 	We wish to use the Implicit Function Theorem and therefore examine the following operator, the ``partial differential with respect to the second argument'':
	\[
	 \Func{T}{E_{12n} }{E_{12n} }{\direc}{h_n'(\gamma_1,\gamma_2)(0 ,\direc).}
	\]
	By (b) this can be rewritten as $T:\direc\mapsto\gamma_1'(\gamma_2+\id_{U_{12n}} )(\direc)+\direc.$
	
	Let $\direc\in E_{12n}$ be given. Then we can estimate
	\begin{align*}
	 \supnorm{\left(T - \id_{E_{12n}}\right)(\direc)  }
						&	=	\sup_{x\in U_{12n}}	\Xnorm{\left(T(\direc) - \direc\right)(x)  }
						\\&	=	\sup_{x\in U_{12n}}	\Xnorm{\gamma_1'(\gamma_2(x)+\id_{U_{12n}}(x) )(\direc(x))  }
						\\&	\leq	\sup_{x\in U_{12n}}	\opnorm{\gamma_1'(\gamma_2(x)+x )}\Xnorm{\direc(x)}  
						\\&	\leq				\supnorm{\gamma_1'} \supnorm{\direc}  
						   	=   				\Dnorm{\gamma_1} \supnorm{\direc} 
							\leq 				\frac{1}{2}	 \supnorm{\direc}.
	\end{align*}
	Thus, $\opnorm{T-\id_{E_{12n}}}\leq \frac{1}{2}<1$. Therefore the bounded operator $T$ is invertible, i.e. an isomorphism of Banach spaces.
	
	By the Implicit Function Theorem, there are neigh\-bor\-hoods \hbox{$\Omega_1\subseteq F_n,$}\linebreak
	 \hbox{$\Omega_2\subseteq E_{12n}$} of $\gamma_1$ and $\gamma_2$ respectively, such that $h_n^{-1}\left(\smset{0}\right)\cap (\Omega_1\times\Omega_2)$ is the graph of a $C^\omega$-map from $\Omega_1$ to $\Omega_2$. But by (c), we know that this function has to be a restriction of $\smfunc{\Inversion_n}{\oBallin{\frac{1}{2}}{F_n}{0} }{\oBallin{\frac{1}{6n} }{E_{12n} }{0}  }$. Therefore $\Inversion_n$ is $C^\omega$ in a neighborhood of $\gamma_1$. Since $\gamma_1$ was arbitrary, $\Inversion_n$ is $C^\omega$.
\end{proof}
This proves Theorem~B 
stated in the introduction for $\K=\C$. As mentioned at the beginning of this section, the proof for $\K=\R$ can be copied verbatim from \cite[Corollary 15.11]{MR2310802}.

\section{Ascending unions of Banach Lie groups}						\label{sec_UNION}
In the following let $G_1 \subseteq G_2 \subseteq \cdots$ be an increasing sequence of analytic Banach Lie groups, such that the inclusion maps $\smfunc{j_n}{G_n}{G_{n+1}}$ are analytic.
Our goal is to construct a Lie group structure on the union $G:= \bigcup_{n=1}^\infty G_n$. But before we can define a manifold structure on $G$, first we have to construct the modelling locally convex vector space.

For every $n\in\N$ let $\g_n := \L(G_n)$ be the corresponding Banach Lie algebra. Since every $j_n$ is an injective morphism of Lie groups with exponential function, it is well known that the corresponding morphism of Lie algebras $\smfunc{i_n := \L(j_n)}{\g_n}{\g_{n+1}}$ is injective as well. Therefore we can identify $i_n(\g_n)$ with $\g_n$ and we may then assume that the Lie algebras form an increasing sequence. The union of this sequence will be denoted by $\g := \bigcup_{n=1}^\infty \g_n$. As a directed union of Lie algebras, this is clearly a Lie algebra. We endow it with the locally convex direct limit topology.

NOTE: 
Since we can only deal with Lie groups modeled on Hausdorff spaces, we have to make the assumption that this direct limit is Hausdorff.

By Corollary \ref{cor_POLYNOMIAL}, the Lie bracket $\smfunc{\brac}{\g\times\g}{\g}$ is continuous and therefore $\left(\g,\brac\right)$ becomes a locally convex Lie algebra.
Note: This would already go wrong in general if we considered direct limits of non-normable Lie algebras $\g_n$.
There are examples where the $\g_n$ are \Frechet{} Lie algebras and the resulting Lie bracket fails to be continuous.

Since every group $G_n$ is a Banach Lie group it admits a smooth exponential function. By commutativity of the diagram
\[\xymatrix{ 	G_n \ar[rr]^{j_n}	& &	G_{n+1}		\\
		\g_n\ar[rr]^{i_n} \ar[u]_{\exp_{G_n}}	& &	\g_{n+1}  \ar[u]_{\exp_{G_{n+1}}}	}
\]

we know that every exponential function $\exp_{G_n}$ can be regarded as the restriction of the exponential function $\exp_{G_{n+1}}$ of the following group. This allows us to define 
\[
 \Func{\Exp}{\g}{G}{x\in \g_n}{\exp_n(x)\in G_n}
\]
(Since we do not have a Lie group structure on $G$ yet, it makes no sense to claim that $\Exp$ is the exponential function of $G$, but it will turn out to the right exponential function.)
So far we did not use the norms on the Banach Lie algebras $\g_n$. In a Banach Lie algebra one usually expects the bilinear map $\smfunc{\brac_n}{\g_n\times\g_n}{\g_n}$ to have a norm less than or equal to $1$, in which case we call $\norm{\cdot}_n$ compatible. This can always be achieved by replacing the norm $\norm{\cdot}_n$ by a scalar multiple. For what follows it will be necessary that all bonding maps $\smfunc{i_n}{\g_n}{\g_{n+1}}$ have norm $\leq 1$. Unfortunately, in general one cannot have both.
There are cases where it is not possible to find equivalent norms such that both, the bonding maps and the Lie brackets, have a norm at most $1$.
Therefore, this has to be made another assumption in Theorem~C from the Introduction.

\begin{proof}[Proof of Theorem C]
 Set $R:=\log\frac{3}{2}$ and $C:=\log2$.
 It is known that in a Banach Lie algebra $\g_n$ with compatible norm $\nnorm{\cdot}$,
 the \BCH{}-series converges for all $x,y\in\g_n$ with $\nnorm{x}+\nnorm{y}<\log\frac{3}{2}$ and defines an analytic multiplication:
\[
 \smfunc{*_n}{\oBallin{R}{\g_n\times\g_n}{0} }{\oBallin{C}{\g_n}{0}}.
\]
We give the space $E_n:= \g_n\times\g_n$ the norm $\norm{(x,y)}_{E_n}:= \nnorm{x} + \nnorm{y} $. (see e.g. \cite[Chapter II, \S 7.2, Proposition 2]{MR1728312})
The set $U:=\bigcup_{n\in \N} \oBallin{R}{E_n}{0}$ is an open $0$-neighborhood in the direct limit
$E:=\bigcup_{n\in\N}\left( \g_n\times \g_n\right) \cong \left(\bigcup_{n\in\N} \g_n\right)\times \left(\bigcup_{n\in\N} \g_n\right)$.

In the case where $\K=\C$, we are now ready to apply Theorem~A, since all hypotheses are satisfied and therefore the map $\smfunc{*=\bigcup_{n\in\N}*_n }{U}{\g}$ is complex analytic.

Since this step does not work in the real case, we have to do some extra work:
If $\K=\R$, we may consider the complexifications of the Lie algebras. Complexifications are again Lie algebras, the norms extend to compatible norms and the bonding maps remain operators of norm at most $1$. Therefore we can apply Theorem~A  to these Lie algebras and obtain a complex analytic \BCH{}-multiplication which we may then restrict to the original real Lie subalgebra.

Having established that the \BCH{}-multiplication is analytic, we can now construct the Lie group structure using Corollary \ref{cor_local_data}:

By hypothesis (c), we know that the exponential function is injective on some neighborhood $V\subseteq \g$. Since $*$ is continuous, there exists a smaller $0$-neighborhood $U'\subseteq U$ such that $U' * U' \subseteq V$. Then, by Corollary \ref{cor_local_data} there exists an analytic Lie group structure on $\generatedby{\Exp(U')}$ making $\Exp$ a diffeomorphism around $0$.
The group $\generatedby{\Exp(U')}$ is equal to 
\[
	\generatedby{\Exp(\g)}= \bigcup_{n\in\N} \generatedby{\Exp_{G_n}(\g_n)}=\bigcup_{n\in\N} \component{G_n}
\]
 which is the union of the identity components of the Banach Lie groups we started with.

We now can extend this manifold structure from $\bigcup_{n\in\N} \component{G_n}$ to the whole group $G$, using Proposition \ref{prop_local_data}.
In fact, being a subgroup, $\bigcup_{n\in\N} \component{G_n}$ is symmetric and contains $1$. As $\bigcup_{n\in\N} \component{G_n}$ already is a Lie group, multiplication and inversion are $C^\omega$ as required. It only remains to show that conjugation with elements $g\in G$ is $C^\omega$.

Let $g\in G$ be such an element. Then there is an $m\in\N$ such that $g\in G_m$. We have to show that $\smfunc{c_g}{\bigcup_{n\in\N} \component{G_n}}{\bigcup_{n\in\N} \component{G_n}}$ is analytic.

Since $\smfunc{ \Ad_g^G:=\bigcup_{n\geq m} \Ad_g^{G_n} }{\bigcup_{n\ge m} \g_n}{ \bigcup_{n\geq m} \g_n }$ is continuous by the locally convex direct limit property, $\Exp$ is a local diffeomorphism at $0$ and $c_g\circ \Exp=\Exp\circ\Ad_g^G$, it follows that $c_g$ is analytic on some identity neighborhood. This is sufficient for a group homomorphism to be analytic everywhere.

This turns $G$ into a $C^\omega$-Lie group in which $\bigcup_{n\in\N} \component{G_n}$ is an open connected subgroup, hence the identity component.

The uniqueness of the manifold structure is clear since $\Exp$ is a local diffeomorphism.
\end{proof}

\section{Example: Lie groups associated with Dirichlet series}						\label{sec_DIRICHLET}
In this section we construct more ``exotic'' examples of Lie groups modelled on Banach and LB-spaces. In fact, the discussion of these examples originally led to the discovery of Theorem~A. All vector spaces and Lie groups will be over the field $\C$.
\subsection*{Banach spaces of Dirichlet series}
\begin{definition}
\begin{enumerate}
 \item  A \emph{Dirichlet series} with values in a complex Banach space $X$ is a formal series of the form
	\[
	 \Din{a_n},
	\]
	where all $a_n$ are elements in $X$. 
 \item	For every $s\in\R$, let $\Halfo{s}:=\set{z\in\C}{\Re(z) >s}$ denote the corresponding open and $\Halfc{s}:=\set{z\in\C}{\Re(z) \geq s}$
	the closed half plane in $\C$.
 \item A Dirichlet series is said to \emph{converge absolutely} on $\Halfc{s}$ if
	\[
	 \DiNorm{\Din{a_n}}{s}:= \sum_{n=1}^\infty \norm{a_n} n^{-s}<\infty.
	\]
 	The space of all Dirichlet series that converge absolutely on $\Halfc{s}$ will be denoted by $\Di{s}{X}$. Together with the norm just defined this vector space becomes a Banach space isomorphic to $\lone(\N,X)$ via the isomorphism
	\[
	 \nnfunc{\Di{s}{X}}{\lone(\N,X)}{\Din{a_n}}{\seqn{a_n\cdot n^{-s}}}.
	\]
 \item The Banach space $X$ can be embedded isometrically into $\Di{s}{X}$ via
	\[
	 \nnfunc{X}{\Di{s}{X}}{a}{\DiOne{a}=\Din{\delta_{n,1} a}	.}
	\]
	All Dirichlet series obtained in this fashion are called \emph{constant}.
\end{enumerate}
\end{definition}

Every Dirichlet series in $\Di{s}{X}$ can be viewed as a continuous bounded function from the closed right half plane
	$\Halfc{s}$ to $X$. In fact, this interpretation defines a bounded operator between Banach spaces of norm $1$:
	\[
	 \nnFunc{\Di{s}{X}}{\left( \BC{ \Halfc{s} }{ X }{ 0 },\supnorm{\cdot} \right) }{\Din{a_n}}{\left( z \mapsto \sum_{n=1}^\infty a_n n^{-z} \right)} \tag{$*$}\label{label_Func}
	\]
 	All functions obtained in this fashion are complex analytic on the open half plane 
	$\Halfo{s}$.
 	Constant Dirichlet series as defined above are mapped to constant functions.
	This assignment is injective which means that it is possible to reconstruct the coefficients $\seqn{a_n}$ from the function. For example, the first coefficient $a_1=\lim_{\Re(z)\rightarrow +\infty}\gamma(z)$. Similarly, the other coefficients may be calculated. This means that a continuous function can have at most one Dirichlet series representation. 

	But there are lots of functions which cannot be written as a Dirichlet series, although they are continuous, bounded on $\Halfc{s}$ and complex analytic on $\Halfo{s}$, e.g. $f(z)=e^{-z}\cdot a$ for an element $a\in X, a\neq 0$. This means the operator is far from being surjective.

Henceforth, we identify a Dirichlet series $\gamma\in\Di{s}{X}$ with the corresponding function from $\Halfc{s}$ to $X$.

\subsection*{LB-spaces of Dirichlet series}
So far, the number $s$ defining the complex half plane $\Halfc{s}$ was fixed. Now, we are interested in Dirichlet series which converge absolutely on some half plane.
\begin{lemma}
 For $s<t$ the bonding maps
$\smfunc{i_s}{\Di{s}{X}}{\Di{t}{X}}$ are bounded operators of norm $\leq1$.
\end{lemma}
\begin{proof} \[  \biggDiNorm{\Din{a_n} }{t}\!\! = 	\sum_{n=1}^\infty \Xnorm{a_n}n^{-t} 
					\leq 	\sum_{n=1}^\infty \Xnorm{a_n}n^{-s} 
					= 	\biggDiNorm{\Din{a_n} }{s}. \qedhere\]
\end{proof}

Since $(\N,\leq)$ is cofinal in $(\R,\leq)$ is suffices to look only at $s\in\N$. So, again, we are dealing with a countable direct limit:
\begin{proposition}									\label{prop_direchlet_limit}
The space
\[
 \Dinf{X}:=\bigcup_{s\in\N}\Di{s}{X},
\]
endowed with the locally convex direct limit topology is Hausdorff and compactly regular.
\end{proposition}
\begin{proof}
 Let $\func{f_s}{ \Di{s}{X} }{X^\N}{\Din{a_n} }{\seqn{a_n} }$ be the map that assigns to every Dirichlet series its sequence of coefficients. This map is continuous since the range space has the product topology and every component of $f_s$ is a continuous functional. The space $X^\N$ is locally convex (it is in fact a \Frechet{} space) and therefore, by the universal property of the locally convex direct limit, there is a continuous extension $\smfunc{f}{\Dinf{X}}{X^\N}$. Since $f$ is injective by construction and $X^\N$ is Hausdorff, it follows that also $\Dinf{X}$ is Hausdorff.

Proposition \ref{prop_compact_regularity} guarantees compact regularity of the limit $\Dinf{X}$ if we can show that for every $s\in\N$ there is a $t\geq s$ and an open $0$-neighborhood $\Omega\subseteq \Di{s}{X}$ such that $\Di{t}{X}, \Di{t+1}{X}, \ldots$ induce the same topology on $\Omega$.

For every given $s\in\N$ we set $t:=s+2$ and $\Omega := \oBallin{1}{\Di{s}{X} }{0}$.
Let $u\ge t$. To see that the topologies on $\Omega$ induced by $\Di{t}{X}$ and $\Di{u}{X}$ agree, it suffices to show that
\[
 \nnfunc{\Omega\subseteq \Di{u}{X}  }{\Di{t}{X} }{\gamma}{\gamma}
\]
is continuous. To this end, let $\epsilon>0$. Since the positive series $\sum_{n=1}^\infty \frac{1}{n^2}$ converges, there is an $n_0\in\N$ such that $\sum_{n>n_0}^\infty \frac{1}{n^2}<\frac{\epsilon}{4}$. Set $\delta:= n_0^{t-u}\cdot\frac{\epsilon}{2}$.
We show that, for any Dirichlet series $\gamma_1,\gamma_2\in \Omega$ with $\DiNorm{\gamma_1-\gamma_2}{u}<\delta$, we have $\DiNorm{\gamma_1-\gamma_2}{t}<\epsilon$.
Since $\gamma_1,\gamma_2\in\Omega$, we have $\gamma_d := \gamma_1-\gamma_2\in 2\Omega$. Therefore $\DiNorm{\gamma_d}{s}<2$ and $\DiNorm{\gamma_d}{u}<\delta$.
Writing $\gamma_d = \Din{a_n}$, we obtain
\begin{align*}
 \DiNorm{\gamma_d}{t}	&=	\sum_{n=1}^\infty \norm{a_n}n^{-t}
			=	\sum_{n\leq n_0} \norm{a_n}n^{-t}+\sum_{n>n_0} \norm{a_n}n^{-t}
		\\&	=	\sum_{n\leq n_0} \norm{a_n}n^{-u}\cdot \underbrace{n^{u-t}}_{\leq n_0^{u-t}}
				  + \sum_{n>n_0} \norm{a_n}n^{-s}\cdot \underbrace{n^{s-t}}_{=\frac{1}{n^2}}
		\\&	\leq	n_0^{u-t}\sum_{n\leq n_0} \norm{a_n}n^{-u}
				  + \sum_{n>n_0}   \underbrace{ \norm{a_n}n^{-s} }_{\leq \DiNorm{\gamma_d}{s}}         \cdot \frac{1}{n^2}
		\\&	\leq	n_0^{u-t} \underbrace{\DiNorm{\gamma_d}{u}}_{<\delta}
				  + \underbrace{\DiNorm{\gamma_d}{s}}_{<2} \underbrace{\sum_{n>n_0}\frac{1}{n^2}}_{<\frac{\epsilon}{4} }
		   	<	n_0^{u-t}\cdot\delta + 2\cdot \frac{\epsilon}{4}
			=	\epsilon.
\end{align*}
This is what we had to show.
\end{proof}

\subsection*{Lie groups associated with Dirichlet series}
From now on, let $G$ denote a fixed complex Banach Lie group with Lie algebra $\g$. As before, $s\in\R$ is a real number.
We know that $G$ has an exponential function $\smfunc{\exp_G}{\g}{G}$. Every Dirichlet series $\gamma\in\Di{s}{\g}$ with values in $\g$ can be composed with the exponential function and yields a continuous function from $\Halfc{s}$ to $G$. All these continuous functions generate a group (with respect to pointwise multiplication of functions):

\begin{theorem}[Lie groups associated with Dirichlet series (Banach case)]
 Let $s\in\R$, a Banach Lie group $G$ with Lie algebra $\g$ be given. Then there exists a unique Banach Lie group structure on the group 
\[
 \Di{s}{G}:= \generatedby{\set{\exp_G \circ \gamma}{\gamma\in \Di{s}{\g}}} \leq \Cont{\Halfc{s}}{G}
\]
 such that
\[
 \func{\Exp_s}{\Di{s}{\g}   }{ \Di{s}{G} }{\gamma}{\exp_G\circ \gamma}.
\]
becomes a local diffeomorphism around $0$.
\end{theorem}
\begin{proof}
We start by choosing a compatible norm on $\g$, i.e. $\gnorm{[x,y]}\leq \gnorm{x}\gnorm{y}$ for all $x,y\in\g$.
Then the space $\Di{s}{\g}$ also carries a continuous bilinear map of operator norm at most $1$:
\[
 \Func{[\cdot,\cdot]}{\Di{s}{\g}\times\Di{s}{\g}}{\Di{s}{\g}}
	{\!\!\!\displaystyle{\left(  \Big(\Din{a_n}\!\Big)\! ,\! \Big(\Din{b_n}\!\Big) \right)}\!\!\!}
	{\displaystyle{\DiN{\Bigg(\!\sum_{\substack{(n_1,n_2)\in \N\times\N\\ n_1\cdot n_2 = N}} \!\!\![a_{n_1},b_{n_2}]\Bigg)       },}}
\]
 turning it into a Banach Lie algebra.
Note that the inner sum is finite. This Lie bracket corresponds to the pointwise Lie bracket of functions.
The Lie algebra $\g$ becomes a closed Lie subalgebra of $\Di{s}{\g}$ by identifying elements of $\g$ with constant Dirichlet series.

Now, having transferred the Banach Lie algebra structure from $\g$ to $\Di{s}{\g}$, we would like to do the same with the group structure.

It is known that (see e.g. \cite[Chapter II, \S 7.2, Proposition 1]{MR1728312}) in a Banach Lie algebra with compatible norm, the \BCH{}-series converges on 
\[
 \Omega_\g:=\set{(x,y)\in\g\times\g}{\norm{x}+\norm{y}<\log2}
\]
and defines an analytic multiplication: $ \smfunc{*}{\Omega_\g}{\g}.$
Since $\Di{s}{\g}$ is a Banach Lie algebra in its own right, we also have a \BCH{}-multiplication there:
$ \smfunc{*}{\Omega_{\Di{s}{\g}}}{\Di{s}{\g}}.$
The \BCH{}-series is defined only in terms of iterated Lie brackets. Since addition and Lie bracket of Dirichlet series in $\Di{s}{\g}$ correspond to the pointwise operations in $\g$, the \BCH{}-multiplication in $\Di{s}{\g}$ corresponds to the pointwise \BCH{}-multiplication of functions.

Since $G$ is a Banach Lie group, it is locally exponential, therefore there is a number $\injepsilon>0$ such that $\exp_G|_{\oBallin{\injepsilon}{\g}{0}}$ is injective. 
Since the \BCH{}-multiplication on $\g$ is continuous, there is a $\delta>0$ such that $\oBallin{\delta}{\g}{0}\times\oBallin{\delta}{\g}{0}\subseteq \Omega_\g$ and $\oBallin{\delta}{\g}{0}*\oBallin{\delta}{\g}{0}\subseteq \oBallin{\injepsilon}{\g}{0}$.

Let $\Cont{\Halfc{s}}{G}$ be the (abstract) group of all continuous maps from $\Halfc{s}$ to $G$ with pointwise multiplication.
Then we may define the following map
\[
 \func{\Exp_s}{\Di{s}{\g}   }{ \Cont{\Halfc{s}}{G} }{\gamma}{\exp_G\circ \gamma}.
\]
The restriction of $\Exp_s$ to $\oBallin{\injepsilon}{\Di{s}{\g}}{0}$ is injective since $\exp_G|_{\oBallin{\injepsilon}{\g}{0}}$ is injective. Here we use that two Dirichlet series in $\Di{s}{\g}$ are equal if they represent the same function in $\Contb{\Halfc{s} }{\g}$.

Now, all hypotheses for Corollary \ref{cor_local_data} are satisfied for $U:=\oBallin{\delta}{\Di{s}{\g} }{0}, V:=\oBallin{\injepsilon}{\Di{s}{\g} }{0}$ and $H:=\Cont{\Halfc{s}}{G}$. Therefore, by Corollary \ref{cor_local_data}, we get a unique $C^\omega$-Lie group structure on the group $\generatedby{\Exp_s( U  )}$ such that 
\[
 \smfunc{\Exp_s|_U}{U\subseteq \Di{s}{\g} }{ \generatedby{\Exp_s( U  )} }
\]
 is a $C^\omega$-diffeomorphism.

But this group, that now has a Lie group structure, is exactly the group $\Di{s}{G}:= \generatedby{\set{\exp_G \circ \gamma}{\gamma\in \Di{s}{\g}}}$ defined above. This is so because for every generator $\exp_G \circ \gamma$ with $\gamma\in \Di{s}{\g}$ there is an $n\in\N$ such that $\frac{1}{n}\gamma\in U$ and therefore 
\[
 \exp_G \circ \gamma= \exp_G \circ \left(n\cdot \frac{1}{n}\gamma \right)=\left(\exp_G\circ \left(\frac{1}{n}\gamma \right)\right)^n\in \generatedby{\Exp_s(U) }.
\]
\end{proof}

\begin{theorem}[Lie groups associated with Dirichlet series (LB case)]
 On the group 
\[
  \Dinf{G}:= \bigcup_{s\in\R}\Di{s}{G}=\generatedby{\set{\exp_G\circ \gamma}{\gamma \in \Di{s}{\g}, s\in \R}}
\]
 there is a unique Lie group structure turning
\[
 \func{\Exp:=\bigcup_{s\in\R}\Exp_s}{\Dinf{\g}}{\Dinf{G}}{\gamma\in \Di{s}{\g}}{\exp_G\circ\gamma}
\]
into a local diffeomorphism around $0$.
\end{theorem}
\begin{proof}
 We wish to use Theorem~C.
 For every $s\in \N$, set $G_s := \Di{s}{G}$. The bonding maps $\smfunc{j_s}{G_s}{G_{s+1} }$ are group homomorphisms.
 Since $j_s\circ\Exp_s=\Exp_s\circ i_s$ with the continuous linear inclusion map $\smfunc{i_s}{\Di{s}{\g} }{\Di{s+1}{\g} }$, we see that each $j_n$ is analytic with $\L(j_s)=i_s$.

 By construction, the norms on the Lie algebras $\Di{s}{\g}$ and the bounded operators 
 $\smfunc{i_s}{\Di{s}{\g}}{\Di{s+1}{\g} }$ have operator norm at most $1$. The locally convex direct limit is Hausdorff by Proposition \ref{prop_direchlet_limit}, and the exponential map $\Exp=\bigcup_{s\in\N}\Exp_s$
 is injective on the $0$-neighborhood $\bigcup_{s\in\N}\oBallin{\injepsilon}{\Di{s}{\g} }{0}$.
Hence, by Theorem~C, there is a unique complex analytic Lie group structure on $G$ such that $\Exp$ is a local diffeomorphism at $0$.
\end{proof}

This proves Theorem~D stated in the introduction.

\addcontentsline{toc}{section}{References}

\bibliography{literatur}{}

\begin{thebibliography}{1}

\bibitem{MR0313811}
Jacek Bochnak and J{\'o}zef Siciak.
\newblock Analytic functions in topological vector spaces.
\newblock {\em Studia Math.}, 39:77--112, 1971.

\bibitem{MR0313810}
Jacek Bochnak and J{\'o}zef Siciak.
\newblock Polynomials and multilinear mappings in topological vector spaces.
\newblock {\em Studia Math.}, 39:59--76, 1971.

\bibitem{MR1728312}
Nicolas Bourbaki.
\newblock {\em Lie groups and {L}ie algebras. {C}hapters 1--3}.
\newblock Elements of Mathematics (Berlin). Springer-Verlag, Berlin, 1998.
\newblock Translated from the French, Reprint of the 1989 English translation.

\bibitem{gloecknerdahmen}
Rafael Dahmen and Helge Gl{\"o}ckner.
\newblock Regularity in {M}ilnor's sense for direct limits of
  infinite-dimensional {L}ie groups.
\newblock {\em (in preparation)}.

\bibitem{MR2310802}
Helge Gl{\"o}ckner.
\newblock Direct limits of infinite-dimensional {L}ie groups compared to direct
  limits in related categories.
\newblock {\em J. Funct. Anal.}, 245(1):19--61, 2007.

\bibitem{MR0436234}
Domingos Pisanelli.
\newblock An example of an infinite {L}ie group.
\newblock {\em Proc. Amer. Math. Soc.}, 62(1):156--160 (1977), 1976.

\bibitem{ArtikelBoris}
Boris Walter.
\newblock Weighted diffeomorphism groups of {B}anach spaces and weighted
  mapping groups.
\newblock {\em (in preparation)}.

\bibitem{MR0418164}
John~C. Wells.
\newblock Invariant manifolds on non-linear operators.
\newblock {\em Pacific J. Math.}, 62(1):285--293, 1976.

\bibitem{MR1977923}
Jochen Wengenroth.
\newblock {\em Derived functors in functional analysis}, volume 1810 of {\em
  Lecture Notes in Mathematics}.
\newblock Springer-Verlag, Berlin, 2003.

\end{thebibliography}
\bibliographystyle{plain}
\end{document}